\documentclass[11pt]{amsart}
\usepackage{amsmath,amssymb,epsfig}
\usepackage{mathrsfs}
\usepackage{xcolor}

\usepackage[colorlinks,
            linkcolor=blue,
            anchorcolor=blue,
            citecolor=blue,
]{hyperref}
\usepackage{cleveref}

\theoremstyle{plain}
\newtheorem{lemma}{Lemma}[section]

\newtheorem{corollary}[lemma]{Corollary}

\newtheorem{proposition}[lemma]{Proposition}
\newtheorem{definition}[lemma]{Definition}

\theoremstyle{remark}
\newtheorem{remark}{Remark}

\allowdisplaybreaks

\def\ee{\varepsilon}

\def\mi{\mathrm{i}}
\def\odd{\mathrm{odd}}
\def\even{\mathrm{even}}

\def\mb{\mathbb}

\begin{document}
 
\title{The Restricted Three-Body Problem as a Perturbed Duffing Equation}

\author{Rongchang Liu}
\address[Rongchang Liu]{School of Mathematics \\ Sichuan University  \\ Chengdu, Sichuan 610064, PR China}
\email{rcliu@scu.edu.cn}

\author{Qiudong Wang}
\address[Qiudong Wang]{Department of Mathematics \\ University of Arizona \\ Tucson, AZ 85750, USA}
\email{wangq@arizona.edu}


\thanks{This work was supported by the Fundamental Research Funds for
the Central Universities.}

\begin{abstract} 
This paper investigates the restricted circular planar three-body problem.  We prove that for every negative Jacobi constant of sufficiently large magnitude, the surface of unperturbed parabolic solutions breaks to induce homoclinic tangles for all but at most finitely many mass ratios of primaries. 

Our approach relies on reformulating the restricted three-body problem as a perturbed Duffing equation 
by introducing new coordinate transformations. This enables us to capture the order of the singularities involved through a renormalized Duffing Hamiltonian structure, to derive the existence and analytic
dependence of the invariant manifolds on the mass ratio of the primaries, and to evaluate the Poincar\'e-Melnikov integral directly through the explicit homoclinic solution of the unperturbed Duffing equation.

Compared to existing works, our proof is significantly simpler and more direct. Moreover, the paper is self-contained: we do not rely on McGehee's analysis in \cite{Mc} to justify the applicability of the Poincar\'e-Melnikov method. 

\end{abstract}

\maketitle

\setcounter{tocdepth}{2}

\section{Introduction}

The planar  three-body problem  consists of  six second-order   equations describing the Newtonian motion of three point masses  $m_1, m_2, m_3 > 0$  in two-dimensional Euclidean space ${\mathbb R}^2$. In contrast to the two-body problem, which is completely integrable, the three-body problem is very difficult to work with partly because it is a system of relatively high phase dimension. 

\vskip0.05in 

The {\it restricted} planar  three-body problem is a  problem of much lower phase dimension that is derived from the planar three-body problem by letting $m_3 \to 0$. In this limit, the motion of $m_1$ and $m_2$ is governed by the two-body problem, independent of $m_3$, while the dynamics of the massless $m_3$ is determined by a reduced system derived from the original equations.
Following tradition \cite{Wintner},  $m_1, m_2$  are called the primaries, and $m_3$ the {\it infinitesimal} mass. The restricted problem models the motion of the infinitesimal mass in the gravitational field generated by the primaries. Since the primaries move along conic sections, each type (circle, ellipse, etc.) leads to a corresponding restricted problem. It has been a convention in celestial mechanics to refer the {\it restricted  circular planar three-body problem} as the restricted three-body problem unless it is otherwise stated. This version admits a conserved quantity known as the Jacobi integral. 
 
\vskip0.05in 

The restricted three-body problem played a central role in Poincar\'e's pioneering development of geometric methods in the theory of ordinary differential equations. Poincar\'e gradually realized that treating the mass ratio of the primaries as a small perturbation parameter leads to the break of the invariant surface formed by parabolic solutions of the unperturbed two-body problem, which gives rise to intricate dynamical structures that he termed \textit{homoclinic tangles} \cite{Po, Po1, Po2}.  He introduced a computational method aimed at verifying the existence of homoclinic tangles in Hamiltonian systems, and applied it to an explicit example, illustrating that homoclinic tangles, as a dynamical phenomenon, do indeed exist. This computational method was later reformulated to cover periodically perturbed Hamiltonian equations by Melnikov \cite{M}.

\vskip0.05in 

However, Poincar\'e \textit{did not} apply this computational scheme to prove that, for the restricted three-body problem, the unperturbed invariant surface formed by parabolic solutions of the two-body problem indeed breaks to form homoclinic tangles. It appears that there were several technical hurdles in applying Poincar\'e's computational scheme (commonly known as the Poincar\'e-Melnikov method in current literature) to the restricted three-body problem. First, the fixed point at infinity, to which the parabolic solutions of the two-body problem approach, is a highly degenerate fixed point, the local solution structure of which is not easily determined. Second, due to this degeneracy, the size of the neighborhood around the fixed point, within which the local dynamics can be fully understood using local analysis techniques, tends to be much smaller than that of a non-degenerate saddle, introducing uncertainty about the validity of applying Poincar\'e's computational scheme. Finally, assuming these two hurdles are somewhat removed, it remains a challenging computational task to explicitly evaluate the Poincar\'e-Melnikov integral for the restricted three-body problem.
 
\vskip0.05in 

The first hurdle was, to some extent, removed much later by a paper of McGehee on local stable and unstable manifolds around certain degenerate fixed points. McGehee proved in \cite{Mc} that, for the restricted three-body problem, the local stable and unstable manifolds of the fixed point at infinity are real analytic in phase space. This study was followed by a paper by Llibre and Sim\'o \cite{LS}, who regarded McGehee's result as a proper justification for applying Poincar\'e's computational scheme. They calculated the corresponding Poincar\'e-Melnikov integral and affirmed that, assuming first that the Jacobi constant is sufficiently large and then that the ratio of the masses of the two primaries is sufficiently small, a homoclinic tangle exists in Poincar\'e's original setting.

\vskip0.05in

In \cite{Xia}, Xia also regarded McGehee's result as a proper justification for the application of the Poincar\'e-Melnikov method in this case. He went one step further and asserted that McGehee's method can also be extended to prove the analytic dependence of the splitting distance of the stable and unstable manifolds on the ratio of the primary masses and on the Jacobi constant. Xia further argued that, due to the existence of singularity of binary collisions, the Melnikov function automatically possesses a simple zero if the Jacobi constant is close to the value allowed for binary collisions. He then concluded that, excluding at most finitely many mass ratios of the primaries, a homoclinic tangle exists in Poincar\'e's original setting.

\vskip0.05in 

In the recent seminal work \cite{G}, Guardia, Mart\'in, and Seara applied techniques from the theory (see \cite{BGFS, G} and references therein) of exponentially small splitting in high-frequency perturbation equations to the restricted three-body problem. By appropriately parametrizing the invariant manifolds via generating functions, they proved that, for any value of the mass ratio of the primaries and sufficiently large Jacobi constants, homoclinic tangles exist.

\vskip0.05in 

We note that the scope of the preceding historical discussion is deliberately narrow: we have only cited results pertaining specifically to the  restricted three-body problem. For a broader historical account of the rich interplay between the study of the N-body problem and the development of the modern  theory of  dynamical systems, we refer the reader to \cite{DH}.

\smallskip

In this paper, we first rewrite the restricted problem explicitly as a  perturbed Duffing equation. By implementing an analytic contraction mapping argument based on a renormalized Duffing Hamiltonian structure, we then construct the invariant manifolds and prove their analytic dependence on the phase and the mass ratio. Furthermore, building on the canonical-coordinate framework of \cite{CW}, we derive integral identities for the primary stable and unstable solutions that allow us to directly evaluate the Melnikov integral using the explicit homoclinic solution of the unperturbed Duffing equation. Our main conclusion is as follows.

\vskip0.05in
\noindent {\bf Main Theorem}. 
{\it There exists $J_0<0$ with sufficiently large magnitude such that, for any Jacobi constant $J<J_0$, and for all but at most finitely many mass ratios of the primaries, the surface of unperturbed parabolic solutions breaks, inducing homoclinic tangles as originally anticipated by Henri Poincar\'e.}

\vskip0.05in

The theorem is a direct consequence of Proposition \ref{props2.1} and Proposition \ref{prop3.1}, whose proof is given at the end of Section~\ref{s3}. Compared to existing works, our proof is significantly simpler and more
direct. Rather than relying on generating functions and the full machinery of exponentially small splitting, we give a direct proof in Duffing coordinates. In particular, we make explicit the analytic dependence, the role of the time reparametrization, and direct evaluation of the Poincar\'e-Melnikov integral, without relying on McGehee's analysis in \cite{Mc}.

\medskip

The paper is organized as follows. In Section~\ref{s1}, we reformulate the restricted three-body problem as a perturbed Duffing equation and derive the governing equation near the unperturbed homoclinic orbit. In Section~\ref{s2}, we establish the existence of the primary stable and unstable solutions and their analytic dependence on phase and mass ratio. In Section~\ref{s3}, we evaluate the Poincar\'e--Melnikov integral to establish transversal intersections for a nontrivial open interval of mass ratios and then prove the finite-exception theorem. Auxiliary computational derivations are collected in Appendices~\ref{s.050501} and \ref{a.051001}.

\section{Derivation of equations around parabolic solution}\label{s1}
This section sets up the basic equations of the paper.
In subsection ~\ref{s1.1}, we derive the equations of the restricted three-body problem using the polar form of the Jacobi coordinates within the Lagrangian formulation of classical mechanics. In Subsection~\ref{s1.3}, we apply the  McGehee coordinates and
rescale time, then make a further coordinate change and use the
Jacobi integral to obtain a \emph{perturbed Duffing equation}. Finally, in subsection ~\ref{s1.4}, we derive the equation for solutions around the homoclinic solution of the unperturbed Duffing equation.

\subsection{Equations for the restricted three-body problem}\label{s1.1}
For the three bodies $m_1, m_2$ and $m_3$ in ${\mathbb R}^2$, let $z_1 = (x_1,   y_1)$ be the vector from $m_1$ to $m_2$ and $z_2 = (x_2, y_2)$ be the vector from the center of masses of $m_1$ and $m_2$ to $m_3$. Denote $z_1 = (x_1, y_1) = x_1+ \mi y_1$ and $z_2 = (x_2, y_2) = x_2 + \mi y_2$.  The variables $z_1$, $z_2$ are the Jacobi coordinates for the three-body problem. Following the convention in classical mechanics, we  use dots on top to represent derivatives with respect to $t$. One dot is for velocity and two dots are for acceleration.
Let 
\[T = \frac{1}{2}\mu_1 |\dot{z}_1|^2 + \frac{1}{2}\mu_2 |\dot{z}_2|^2\]
be the kinetic energy where 
\[\mu_1 = \frac{m_1 m_2}{m_1+ m_2}, \ \ \ \ \ \ \ \ \mu_2 = \frac{m_3(m_1+m_2)}{m_1 + m_2 + m_3}.\]
We use polar coordinates for $z_1$ and $z_2$ by letting 
\begin{equation*}
z_1 = r_1 e^{\mi \theta_1}, \ \ \ \ z_2 = r_2 e^{\mi \theta_2}.
\end{equation*}
The kinetic energy then becomes
\begin{equation*}
T = \frac{1}{2} \mu_1 \left(\dot{r}_1^2 + r_1^2 \dot{\theta}_1^2\right) +\frac{1}{2} \mu_2 \left(\dot{r}_2^2 + r_2^2 \dot{\theta}_2^2\right).
\end{equation*}
Let the potential energy be 
\[U = \frac{m_1 m_2}{r_{12}} + \frac{m_1 m_3}{r_{13}} + \frac{m_2 m_3}{r_{23}},\]
where $r_{ij}$ are the distances from $m_i$ to $m_j$. The  Lagrangian for the planar three-body problem is 
\[
L = T + U= \frac{1}{2} \mu_1 \left(\dot{r}_1^2 + r_1^2 \dot{\theta}_1^2\right) +\frac{1}{2} \mu_2 \left(\dot{r}_2^2 + r_2^2 \dot{\theta}_2^2\right) + U.
\]
Denoting 
\begin{equation*}
\alpha_1 = \frac{m_2}{m_1+ m_2}, \ \ \ \ \ \ \ \ \alpha_2 = \frac{m_1}{m_1+ m_2},
\end{equation*}
we have 
\begin{equation*}
\begin{split}
r_{12} = & |z_1| = r_1,\\
r_{13} = & |z_2 + \alpha_1 z_1| = \sqrt{r_2^2 + \alpha_1^2 r_1^2 + 2 \alpha_1 r_1 r_2 \cos (\theta_2 - \theta_1)},\\
r_{23} = & |z_2 - \alpha_2 z_1| = \sqrt{r_2^2 + \alpha_2^2 r_1^2 -2 \alpha_2 r_1 r_2 \cos (\theta_2 - \theta_1)}.
\end{split}
\end{equation*} 

Writing the Euler-Lagrange equations explicitly, we obtain the equation of motion for the planar three-body problem as
\begin{equation*}
\begin{split}
& \ddot{r}_1 = r_1 \dot{\theta}^2_1 + \mu_1^{-1}\partial_{r_1} U, \ \ \ \ \ddot{\theta}_1 = - \frac{2\dot{r}_1 \dot{\theta_1}}{r_1}+ \mu_1^{-1}\frac{1}{r_1^2} \partial_{\theta_1} U,\\
& \ddot{r}_2 = r_2 \dot{\theta}_2^2 + \mu_2^{-1} \partial_{r_2}U, \ \ \ \ \ddot{\theta}_2 = - \frac{2 \dot{r}_2 \dot{\theta}_2}{r_2} + \mu_2^{-1}\frac{1}{r_2^2}\partial_{\theta_2} U.
\end{split}
\end{equation*}
 
\vskip0.05in

To obtain the equations for the restricted three-body problem, we first divide the equations for the third body by its reduced mass and then take the limit $m_3\to0$. We normalize
\[m_1 + m_2 = 1.\]
The motion of the primaries is independent of the infinitesimal body; for the circular solution the equations for $r_1$, $\theta_1$ become 
$$
\ddot{r}_1 = r_1 \dot{\theta}^2_1 - \frac{1}{r_1^2}, \ \ \ \ \ddot{\theta}_1 = - \frac{2\dot{r}_1 \dot{\theta_1}}{r_1}
$$
which allows a specific solution
$$
r_1 = 1, \ \ \ \ \theta_1 = t.
$$
We adopt this specific solution to write the equations for $r_2, \theta_2$ as 
\begin{equation}\label{diff-eqn2}
\ddot{r}_2 = r_2 \dot{\theta}_2^2 + f, \ \ \ \ \ddot{\theta}_2 = - \frac{2 \dot{r}_2 \dot{\theta}_2}{r_2} + g,
\end{equation}
where
\begin{align*}
f &= - \frac{m_1 \left(r_2 +m_2 \cos (\theta_2 -t)\right)}{r_{13}^3} - \frac{m_2 \left(r_2 -m_1 \cos (\theta_2 - t)\right)}{r_{23}^3},\\
g & = \frac{m_1 m_2 \sin (\theta_2-t)}{r_2}\left( \frac{1}{r_{13}^3} - \frac{1}{ r_{23}^3}\right),\\
r_{13} &= \sqrt{r_2^2 + m_2^2 + 2 m_2 r_2 \cos (\theta_2 - t)},\\
r_{23} &= \sqrt{r_2^2 + m_1^2 -2 m_1 r_2 \cos (\theta_2 -t)}.
\end{align*}

\vskip0.05in

System \eqref{diff-eqn2} represents the restricted three-body problem, which describes the motion of a massless particle under the gravitational influence of two finite masses moving in circular orbits in a two-dimensional physical space.

\vskip0.05in

It is straightforward to check the following well-known fact in the above formulation. 

\begin{lemma}
	For all solutions of equation \eqref{diff-eqn2}, 
	\begin{equation}\label{jacobi-int}
J := \frac{1}{2}\left(\dot{r}_2^2 + r_2^2 (\dot{\theta}_2 -1)^2 \right)- \frac{1}{2} r_2^2 - m_1 r_{13}^{-1} - m_2 r_{23}^{-1}
\end{equation}
	is an integral constant, which is commonly referred to as the Jacobi constant.
\end{lemma} 

\subsection{Reduction to a perturbed Duffing equation}\label{s1.3} The phase variables for equation ~\eqref{diff-eqn2} are
$(r_2,   \dot{r}_2,  \theta_2,  \dot{\theta}_2)$. We now introduce new phase variables $(u, v, \theta, w)$ and a new time $\tau$, following McGehee, by letting 
\begin{equation*}
u = r_2^{-1}, \ \ v = u^{-1/2} \dot{r}_2, \ \ \theta = \theta_2 -t , \ \ w = u^{-3/2}\dot{\theta}_2, \ \ \ d \tau = \frac{1}{\sqrt{2}}u^{3/2} d t.
\end{equation*}
Equations for $u, v, \theta, w$ in $\tau$ are
\begin{align}\label{eq:uwvtheta}
\begin{split}
\frac{du}{d\tau} = & - \sqrt{2} v u,\qquad \frac{d\theta}{d\tau} = \sqrt{2} \left(w - u^{-3/2}\right),\\
\frac{dw}{d\tau} = &-\frac{1}{ \sqrt{2}} v w + \sqrt{2}u G,\qquad \frac{dv}{d\tau} = \frac{1}{ \sqrt{2}} v^2 + \sqrt{2} w^2 - \sqrt{2} + \sqrt{2} F,
\end{split}
\end{align}
where
\begin{equation*}
\begin{split}
F = & 1 - m_1 \left(1 +m_2 u \cos \theta\right)R_{13}^{-3} - m_2 \left(1 -m_1 u \cos \theta \right)R_{23}^{-3},\\
G = & m_1 m_2 \sin \theta \left(R_{13}^{-3} - R_{23}^{-3}\right),
\end{split}
\end{equation*}
in which 
\begin{equation*}
\begin{split}
R_{13} = & \sqrt{1+ m_2^2 u^2 + 2 m_2 u \cos \theta }, \ \ \ \ R_{23} = \sqrt{1 + m_1^2 u^2 -2 m_1 u \cos \theta }.
\end{split}
\end{equation*}
We further introduce new variables $X, Y$ by letting 
\begin{equation}\label{XY}
X = w, \ \ \ Y = -\frac{1}{\sqrt{2}}vw.
\end{equation}
The equations for $\theta, X, Y$ in $\tau$ are 
\begin{align}\label{EqnXY}
\begin{split}
\frac{d\theta}{d\tau} = & \sqrt{2} \left(X- u^{-3/2}\right),\\
\frac{dX}{d\tau} = &Y + \sqrt{2} u G,\qquad \frac{dY}{d\tau}= X - X^3 - XF+ \sqrt{2} u\frac{Y}{X} G.
\end{split}
\end{align}
Here we dropped the equation for $u$ despite a clear occurrence of $u$ on the right hand of the equations for $\theta, X, Y$. However, note that we can solve  $u$ for $\theta, X, Y$ by using the Jacobi integral  \eqref{jacobi-int},
\begin{equation}\label{e.050304}
u^{1/2}J = \frac{1}{2} u^{3/2} v^2 + \frac{1}{2} u^{3/2} w^2 - w - m_1u^{3/2} R_{13}^{-1} - m_2 u^{3/2} R_{23}^{-1}.
\end{equation}
 
In this paper, we only consider the case of $J < 0$ and assume $|J|\gg 1$. We now replace $u$ with a new variable $\mathcal{J}$ by letting
\[\mathcal{J} = |J| u^{1/2} X^{-1},\]
and set 
\[|J|^{-1} = \varepsilon, \ \ \ \ m_2 = \rho, \ \ \ \ m_1 = 1- \rho.\]
Then equation \eqref{EqnXY} for the restricted three-body problem can be rewritten into a perturbed Duffing equation as 
\begin{align}\label{EqnXYU1}
\begin{split}
\frac{d\theta}{d\tau} = & \sqrt{2} \left(X- \frac{1}{\varepsilon^{3} \mathcal{J}^{3}X^{3} }\right),\\
\frac{dX}{d\tau} = &Y + \sqrt{2} \varepsilon^2 \mathcal{J}^2 X^2 G,\qquad \frac{dY}{d\tau}= X - X^3 - XF+ \sqrt{2}\varepsilon^2 \mathcal{J}^2 X Y G.
\end{split}
\end{align}
Here 
\begin{align}\label{e.050302}
\begin{split}
F = & 1 - (1-\rho) \left(1 +\rho \varepsilon^2 \mathcal{J}^2 X^2 \cos \theta\right)R_{13}^{-3} - \rho \left(1 -(1-\rho) \varepsilon^2 \mathcal{J}^2 X^2 \cos \theta \right)R_{23}^{-3},\\
G = & \rho(1-\rho) \sin \theta \left(R_{13}^{-3} - R_{23}^{-3} \right),
\end{split}
\end{align}
where
\begin{align}\label{e.050403}
\begin{split}
R_{13} = & \sqrt{1+\rho^2\varepsilon^4 \mathcal{J}^4 X^4 + 2 \rho \varepsilon^2 \mathcal{J}^2 X^2 \cos \theta },\\
R_{23} = & \sqrt{1 + (1-\rho)^2 \varepsilon^4 \mathcal{J}^4 X^4 -2 (1-\rho) \varepsilon^2 \mathcal{J}^2 X^2 \cos \theta },
\end{split}
\end{align}
and $\mathcal{J}$, as a function of $X, Y, \theta$, is determined by the equation for the Jacobi integral \eqref{e.050304} in the new coordinates,
\begin{equation}\label{Jacobi1}
\mathcal{J} = 1- \varepsilon^3 \mathcal{J}^3 Y^2 - \frac{1}{2} \varepsilon^3 \mathcal{J}^3 X^4 + (1-\rho) \varepsilon^3 \mathcal{J}^3 X^2 R_{13}^{-1}+ \rho \varepsilon^3 \mathcal{J}^3 X^2 R_{23}^{-1}.
\end{equation}

\subsection{Equations around unperturbed parabolic solution}\label{s1.4}
Since we are mainly interested in the dynamics in a neighborhood of the unperturbed parabolic solution, it is useful to develop the corresponding equations for later use. 
In view of equation \eqref{EqnXYU1}, one observes that when $\ee=0$, the equation for $(X,Y)$ reduces to the Duffing equation with a saddle at the origin. To deal with the singularity at $\ee=0$ in the $\theta$ equation, we regard the following system
\begin{align}\label{e.050301}
\frac{d\theta}{d\tau} = & \sqrt{2} \left(X- \frac{1}{\varepsilon^{3}X^{3} }\right),\qquad \frac{dX}{d\tau} = Y ,\qquad \frac{dY}{d\tau}= X - X^3,
\end{align}
as the ``unperturbed" system to \eqref{EqnXYU1}. A direct integration shows that $(X,Y,\theta)=(a,b,\psi)$ with 
\begin{align}\label{e.050303}
\begin{split}
&a(\tau) = \frac{2 \sqrt{2}}{ e^{\tau}+e^{-\tau}}, \ \ \ b(\tau) = \frac{2 \sqrt{2} \left(e^{-\tau}-e^{\tau}\right)}{\left(e^{\tau}+e^{-\tau}\right)^2},\\
& \psi(\tau) = 2 \tan^{-1} \frac{1}{2}(e^{\tau}- e^{-\tau}) -\frac{1}{48\varepsilon^3} \left(e^{3 \tau}- e^{-3\tau}\right)-\frac{3}{16\varepsilon^3}\left(e^{\tau}- e^{-\tau}\right),
\end{split}
\end{align}
is a homoclinic solution to \eqref{e.050301} with initial condition $(X(0),Y(0),\theta(0))=(\sqrt{2},0,0)$. 

\vskip0.05in

Fixing an arbitrary phase $\theta_0\in \mb R$ and introducing the new variables $(x,y,\Theta)$ by letting 
\begin{align}\label{e.050407}
x= X-a, \, y= Y-b, \, \Theta = \theta - \theta_0-\psi,
\end{align}
we obtain from \eqref{EqnXYU1} the equation for solutions around the homoclinic orbit as 
\begin{align}\label{eqnxyS}
\begin{split}
\frac{d \Theta}{d\tau} &= \frac{\sqrt{2} S}{\varepsilon^3 \mathcal{J}_{\theta_0, \psi}^3 (x+a)^3 a^3},\\
\frac{dx}{d\tau} &= y + P,\qquad \frac{dy}{d\tau} = (1-3a^2) x + Q,
\end{split}
\end{align}
with 
\begin{align}\label{SPQS1}
\begin{split}
S = & x^3 + 3a x^2 +3a^2 x + \left(\mathcal{J}_{\theta_0, \psi}^3 -1\right)(x+a)^3 + \varepsilon^3 \mathcal{J}_{\theta_0, \psi}^3 x (x+a)^3 a^3,\\
P = & \sqrt{2} \varepsilon^2 \mathcal{J}_{\theta_0, \psi}^2 (x+a)^2 G_{\theta_0, \psi},\\
Q = &- x^3 - 3ax^2 - (x+a)F_{\theta_0, \psi} + \sqrt{2}\varepsilon^2 \mathcal{J}_{\theta_0, \psi}^2 (x+a)(y+b) G_{\theta_0, \psi},
\end{split}
\end{align}
where $(a,b,\psi)$ denotes the parabolic solution given in \eqref{e.050303}, and $\mathcal{J}_{\theta_0, \psi}, F_{\theta_0, \psi}, G_{\theta_0, \psi}$ are obtained from \eqref{Jacobi1} and \eqref{e.050302} through substituting  $X = x+a, Y = y+b, \theta=\Theta + \theta_0 + \psi$.

\section{Existence and analyticity of invariant manifolds}\label{s2}
This section proves existence of the primary invariant manifolds and, for each fixed $\varepsilon>0$, real-analytic dependence of their finite section traces on $(\theta_0,\rho)$. This is the parameter dependence needed for the splitting argument in the next section. Let $\varepsilon_0$ be a small positive number and assume 
\begin{align}\label{e.051102}
(\rho, \varepsilon) \in D_{\ee_0}: = (-\varepsilon_0,\varepsilon_0+1/2)\times (0, \varepsilon_0).
\end{align}
This is to say that we impose no restriction on the masses of the primary bodies but assume the Jacobi constant $J < 0$ is such that $|J|$ is large.  {\it We regard $\varepsilon$ as the parameter for perturbation }in this section.  Let  
\begin{align}\label{e.050402}
\ell^+ = \{ (a(\tau), b(\tau)): \ \ \tau \in [0, \ + \infty) \}
\end{align}
be the positive part of the homoclinic solution $(a(\tau), b(\tau))$ on the plane and 
$\mathcal D_{\ell}^+$ be a small neighborhood of $\ell^+ \cup (0, 0)$. 
We also use  $I$ to denote a small segment of the $X$-axis centered at $a(0) = \sqrt{2}$ and let $X_0 \in I$ to study the solution $(X(\tau),  Y(\tau), \theta(\tau))$ of equation \eqref{EqnXYU1} satisfies the initial condition 
\begin{align}\label{iniS} 
(X(0), Y(0),  \theta(0)) = (X_0, 0, \theta_0). 
\end{align}
\begin{definition}\label{def:primary}
 	We say that a solution $(X(\tau), Y(\tau), \theta(\tau)), \ \tau \in [0, +\infty)$ of equation \eqref{EqnXYU1} satisfying \eqref{iniS} is a primary stable solution if 	 
	\begin{itemize}
		\item[(i)]$(X(\tau),Y(\tau))\in {\mathcal D}_{\ell}^+$ and $X(\tau)>0$ for every finite $\tau\ge0$; 
		\item[(ii)]  $(X(\tau), Y(\tau)) \to (0, 0)$ as $\tau \to +\infty$.
	\end{itemize}
\end{definition}

\begin{remark}
Let ${\mathcal D}^-_{\ell}$ be a small neighborhood of the negative branch
$\ell^-$ of $(a(\tau),b(\tau))$. Primary unstable solutions are defined analogously on
$(-\infty,0]$, with initial condition \eqref{iniS}, with
$\mathcal D_\ell^+$ replaced by a small neighborhood
$\mathcal D_\ell^-$ of $\ell^-\cup\{(0,0)\}$, and with convergence
to $(0,0)$ as $\tau\to-\infty$.
\end{remark}

The main result of this section is 
\begin{proposition}\label{props2.1}
There exists $\varepsilon_0>0$ such that, for every
\[
(\theta_0,\rho,\varepsilon)\in \mathbb R\times D_{\varepsilon_0},
\]
 there exists a unique
$X_0^s(\theta_0,\rho,\varepsilon)\in I$ such that the solution $(X^{s}(\tau,\theta_0,\rho,\varepsilon),Y^{s}(\tau,\theta_0,\rho,\varepsilon))$
of \eqref{EqnXYU1} with initial data
$(X_0^s,0,\theta_0)$ is a primary stable solution.
The analogous assertion holds for the primary unstable solution $(X^{u}(\tau,\theta_0,\rho,\varepsilon),Y^{u}(\tau,\theta_0,\rho,\varepsilon))$.

For each fixed $\varepsilon\in(0,\varepsilon_0)$, the section coordinates
$X_0^{s,u}(\theta_0,\rho,\varepsilon)$ are jointly real analytic
in $(\theta_0,\rho)$ on a neighborhood of
$\mathbb R/2\pi\mathbb Z\times[0,1/2]$. 
\end{proposition}
\begin{remark}
For each fixed $\varepsilon\in(0,\varepsilon_0)$ and every compact
interval $K\subset[0,\infty)$, the map
$(\theta_0,\rho)\mapsto(X^s,Y^s)|_K$ is real analytic with values
in $C(K;\mathbb R^2)$ on a neighborhood of
$\mathbb R/2\pi\mathbb Z\times[0,1/2]$.
The analogous assertion holds for the unstable branch when
$K\subset(-\infty,0]$.
Indeed, the section data are analytic by Proposition~\ref{props2.1},
and $X$ stays away from zero on each compact time interval.
The assertion therefore follows from analytic dependence of
the ODE flow in \eqref{EqnXYU1}.
The parameter neighborhood may depend on $K$ and $\varepsilon$.
\end{remark}
The proof of \Cref{props2.1} is given at the end of this section. The main issue is not the analyticity of the finite-time flow, but the
unbounded fast angle near infinity in the original Duffing time $\tau$, which could result in uncontrollable objects when $(\theta_0,\rho)$ are complexified.  To overcome this difficulty, we introduce a  radial coordinate $\chi=\varepsilon^{-1}u^{1/2}$ to synchronize
nearby trajectories and subtract the complete two-body angle.  The
Duffing Hamiltonian then provides a single transverse variable: its evolution
has a linear principal part, and the Jacobi integral recovers the variables $(v,w)$ algebraically.

We organize the argument into three parts. First, we use the radial
variable $\chi$ to derive exact integral
equations for the normalized Duffing Hamiltonian and the phase
correction. Second, we solve these equations by a contraction mapping argument 
and establish joint holomorphic dependence on the crossing phase and the mass ratio. Then we reconstruct
exact stable trajectories from the fixed point and show that they
form a positively invariant local stable surface. Finally, we
continue these trajectories backward to the turning section $Y=0$ with $X\in I$ to
obtain an analytic section graph, and prove that every primary
stable solution coincides with one of the constructed trajectories.
The unstable assertions follow by reversibility.

Throughout, $\varepsilon>0$ is real and fixed when discussing analytic
parameter dependence,  while only the phase $\theta_0$ and the mass ratio $\rho$ are complexified.

\subsection{Exact reduction to an energy--phase system}
Put $\chi=\varepsilon^{-1}u^{1/2}$ and $\sigma=\varepsilon^3t$.
Derivatives in the next equations are taken with respect to $\sigma$.
The original equations \eqref{eq:uwvtheta} and the Jacobi identity \eqref{e.050304} become
\begin{align}
\begin{split}
\frac{d\chi}{d\sigma}&=-\tfrac12\chi^4v,\qquad \frac{d\theta}{d\sigma}=-\varepsilon^{-3}+\chi^3w,\\
\frac{dv}{d\sigma}&=\chi^3\bigl(\tfrac12v^2+w^2-1+F\bigr),\qquad \frac{dw}{d\sigma}=-\tfrac12\chi^3vw+\varepsilon^2\chi^5G,
\end{split}\label{energy-physical-system}\\
w&=\chi+\varepsilon^3\chi^3 \bigl\{\tfrac12(v^2+w^2)-\mathcal V\bigr\}, \label{energy-jacobi}
\end{align}
where $\mathcal V=(1-\rho)R_{13}^{-1}+\rho R_{23}^{-1}$ and
\begin{align*}
R_{13}^2&=1+2\rho\varepsilon^2\chi^2\cos\theta +\rho^2\varepsilon^4\chi^4,\qquad R_{23}^2=1-2(1-\rho)\varepsilon^2\chi^2\cos\theta +(1-\rho)^2\varepsilon^4\chi^4.
\end{align*}

Given any $r_*>0$, we can choose $\varepsilon_0>0$ small such that for $0\le\chi\le1$ the gravitational functions extend holomorphically on the complex neighborhood 
\begin{align}\label{eq:mathcal N}
\mathcal N_{r_*} : = \left\{(\theta,\rho)\in\mathbb C^2: |\operatorname{Im}\theta|<r_*, \, \operatorname{dist}(\rho,[0,1/2])<r_*\right\}.
\end{align}
The barycentric cancellation as in \Cref{s.050501} 
gives bounded holomorphic functions $\mathcal V_4,F_4,G_2$, 
with uniformly bounded first derivatives for all real
$0<\varepsilon<\varepsilon_0$, such that
\begin{align}\label{energy-force-factors}
\mathcal V=1+\varepsilon^4\chi^4\mathcal V_4,\qquad F=\varepsilon^4\chi^4F_4,\qquad G=\varepsilon^2\chi^2G_2.
\end{align}

Recall the Duffing Hamiltonian
$H_D(X,Y)=\tfrac12Y^2-\tfrac12X^2+\tfrac14X^4$.
Since $X=w$ and $Y=-vw/\sqrt2$, we use its normalized defect
\begin{align*}
\mathbb H=\frac{\tfrac12(v^2+w^2)-1}{\varepsilon^4\chi^4} =\frac{2H_D(X,Y)}{\varepsilon^4\chi^4w^2},
\end{align*}
which vanishes on the unperturbed parabolic orbit.  This normalized Duffing energy will serve as the transverse variable. Together with radial synchronization and subtraction of the exact two-body phase, it converts \eqref{EqnXYU1} to an energy--phase system in which the equation for the angular correction has a bounded right-hand side.

Let $0<\delta\le1$ and set $K=\mathbb H-\mathcal V_4$. 
The Jacobi identity
\eqref{energy-jacobi} and 
the definition of $\mathbb H$ then give
\begin{align}\label{energy-explicit-coordinates}
w=\chi+\varepsilon^7\chi^7K,\qquad v=\sqrt{2-w^2+2\varepsilon^4\chi^4\mathbb H},
\end{align}
where we work on the positive radial branch $v>0$. 
For $\mathbb H$ in a fixed bounded complex disc, these are holomorphic
expressions and $v$ stays uniformly away from zero if $\varepsilon_0$
is sufficiently small. Since $\chi=\mathcal{J}X$ and $\mathcal{J}=1$ on the unperturbed parabolic orbit, we have $\chi=a(\tau)$ along that orbit.  Let $T(\chi)$
be the inverse of $a$ from \eqref{e.050303} on the decreasing stable Duffing branch, with
$T_\delta=T(\delta)$.  At the radial section $\chi=\delta$ denote
the actual angle by $\theta_\delta$, and write
\begin{align}\label{eq:subtraction-2bdtheta}
\theta(\chi)=\theta_\delta+\psi(T(\chi))-\psi(T_\delta) +\omega(\chi),\qquad \omega(\delta)=0.
\end{align}
By \eqref{e.050303}, the reference phase satisfies
\begin{align}\label{eq:psichi}
\frac{d}{d\chi}\psi(T(\chi)) =\frac{2\varepsilon^{-3}}{\chi^4v_0(\chi)}-\frac2{v_0(\chi)}, \quad v_0(\chi)=\sqrt{2-\chi^2}.
\end{align}
For brevity in the angular equation, put
\[L=\mathbb H-\varepsilon^3\chi^4K-\tfrac12\varepsilon^{10}\chi^{10}K^2.\]

\begin{proposition}\label{prop:energy-reduction}
For $\varepsilon\in(0,\varepsilon_0)$, on the positive branch $v> 0$ with $\chi$ as the independent variable, system \eqref{EqnXYU1}  is equivalent to
\begin{align}\label{energy-radial-system}
\chi \frac{d\mathbb H}{d\chi}+6\mathbb H=A(\chi,\mathbb H,\omega;\theta_\delta,\rho), \quad \frac{d\omega}{d\chi}=B(\chi,\mathbb H,\omega;\theta_\delta,\rho),
\end{align}
together with \eqref{energy-explicit-coordinates}, where the right-hand sides are evaluated at
$\theta=\theta_\delta+\psi(T(\chi))-\psi(T_\delta)+\omega$ and
\begin{align}\label{energy-B}
A(\chi)=-2\left(F_4+\frac wvG_2\right),\quad B(\chi)=-\frac{4\varepsilon(1-\varepsilon^3\chi^4)L} {vv_0(v+v_0)}-\frac{2\varepsilon^7\chi^6K}{v}.
\end{align}
\end{proposition}

\begin{proof}
By \eqref{energy-physical-system},
\begin{align*}
\frac{d}{d\sigma}\left(\tfrac12(v^2+w^2)-1\right) &=\chi^3v\left(\tfrac12(v^2+w^2)-1+F\right) +\varepsilon^2\chi^5wG.
\end{align*}
Since $d\chi/d\sigma=-\chi^4v/2$,
\eqref{energy-force-factors} gives
\begin{align}\label{eq:energy-chi}
\frac{d}{d\chi} \left[\chi^2\left(\tfrac12(v^2+w^2)-1\right)\right] =-2\varepsilon^4\chi^5 \left(F_4+\frac wvG_2\right).
\end{align}
Substituting
$\tfrac12(v^2+w^2)-1=\varepsilon^4\chi^4\mathbb H$
yields the first equation in \eqref{energy-radial-system}. By \eqref{energy-explicit-coordinates}, 
\begin{align*}
v^2-v_0^2=2\varepsilon^4\chi^4L,\qquad \frac1v-\frac1{v_0} =-\frac{2\varepsilon^4\chi^4L}{vv_0(v+v_0)}.
\end{align*}
The angular equation for $\frac{d\theta}{d\sigma}$ divided by $\frac{d\chi}{d\sigma}$ reads
\begin{align*}
\frac{d\theta}{d\chi}=\frac{2\varepsilon^{-3}}{\chi^4v}-\frac{2w}{\chi v}.
\end{align*}
Subtracting the reference angular derivative \eqref{eq:psichi} and using
$w/\chi=1+\varepsilon^7\chi^6K$ gives 
\begin{align*}
\frac{d \omega}{d \chi}=\left(\frac{2 \varepsilon^{-3}}{\chi^4}-2\right)\left(\frac{1}{v}-\frac{1}{v_0}\right)-\frac{2 \varepsilon^7 \chi^6 K}{v}=B
\end{align*}
which is exactly \eqref{energy-B}.

The conversion from $(\chi,\mathbb H,\omega)$ to $(\chi,v,w,\theta)$ on the positive branch $v>0$ follows direct computation from \eqref{energy-explicit-coordinates}, \eqref{energy-jacobi} and \eqref{energy-radial-system}. 
\end{proof}

\begin{remark}
The key mechanism is the subtraction of the unperturbed two-body
reference orbit at the same radial level $\chi$.
For bounded $\mathbb H$, the Jacobi identity and the
Duffing-energy normalization give
\begin{align*}
v-v_0=O(\varepsilon^4\chi^4),\qquad w/\chi-1=O(\varepsilon^7\chi^6).
\end{align*}
Subtracting also the exact two-body phase from $\theta$ as in \eqref{eq:subtraction-2bdtheta} yields
\begin{align*}
\frac{d\omega}{d\chi} &=\left(\frac{2\varepsilon^{-3}}{\chi^4}-2\right) \left(\frac1v-\frac1{v_0}\right) -\frac{2\varepsilon^7\chi^6K}{v} =O(\varepsilon).
\end{align*}
The decisive cancellation is
\begin{align*}
\underbrace{\varepsilon^{-3}\chi^{-4}}_{ \text{singular angular factor }} \underbrace{\varepsilon^4\chi^4}_{ \text{ radial velocity defect}} =\varepsilon.
\end{align*}
Thus the fourth-order angular singularity disappears from
the equation for $\omega$, providing the key estimate
for the analytic Volterra construction below.
\end{remark}

\subsection{Holomorphic fixed point and the analytic local graph}
We first construct solutions for which $\mathbb H$ remains
bounded as $\chi\to0$. This condition is satisfied by every
primary stable solution, as verified in
Lemma~\ref{lem:primary-asymptotics} below.  Together with
$\omega(\delta)=0$, it converts \eqref{energy-radial-system} into
\begin{align}
\mathbb H(\chi)&=\chi^{-6}\int_0^\chi s^5 A(s,\mathbb H(s),\omega(s);\theta_\delta,\rho)\,ds, \label{energy-Volterra}\\
\omega(\chi)&=-\int_\chi^\delta B(s,\mathbb H(s),\omega(s);\theta_\delta,\rho)\,ds. \label{energy-phase-Volterra}
\end{align}
Denote by $\|\cdot\|_\infty$  the supremum
norm on $(0,\delta]$ in the radial variable $\chi$.
The integrals are absolutely convergent in the function class below. 

\begin{lemma}
\label{lem:energy-local-graph}
There exist $R,C,r_*,\varepsilon_0>0$, independent of
$0<\delta\le1$, such that, for each real
$0<\varepsilon<\varepsilon_0$ and
$(\theta_\delta,\rho)\in\mathcal N_{r_*}$, the Volterra system
\eqref{energy-Volterra}--\eqref{energy-phase-Volterra}
has a unique solution satisfying
\begin{align*}
\|\mathbb H\|_\infty\le R,\qquad \|\omega\|_\infty\le C\varepsilon\delta.
\end{align*}
The solution is jointly holomorphic in $(\theta_\delta,\rho)$
with values in $C_b((0,\delta],\mathbb C)^2$, is
$2\pi$-periodic in $\theta_\delta$, and is real for real
parameters. At $\rho=0$, the fixed point is $(\mathbb H,\omega)=(0,0)$.
For real $\theta_\delta$ and $\rho\in[0,1/2]$, at each fixed
$\chi\in(0,\delta]$, it also satisfies
\begin{align}\label{energy-graph-bounds}
\|\partial_{\theta_\delta}\omega\|_\infty \le C\varepsilon\delta,\quad |w-\chi|+|\partial_\rho w| \le C\varepsilon^7\chi^7,\quad |v-v_0|+|\partial_\rho v| \le C\varepsilon^4\chi^4.
\end{align}
\end{lemma}

\begin{proof}
Choose $r_*>0$ and $\varepsilon_0\le1$ so that the gravitational coefficients $\mathcal V_4,F_4,G_2$ in \eqref{energy-force-factors} are uniformly bounded and holomorphic on $\mathcal N_{2r_*}$, where $\mathcal N_{r_*}$ is defined in \eqref{eq:mathcal N}. Cauchy's estimates give a constant $M$, independent of $0<\varepsilon<\varepsilon_0$ and $0<\chi\le1$, such that
\begin{align}\label{eq:QpartialQ}
|Q|+|\partial_\theta Q|\le M,\qquad Q\in\{\mathcal V_4,F_4,G_2\},\qquad (\theta,\rho)\in\mathcal N_{3r_*/2}.
\end{align}
These three gravitational coefficients are $2\pi$-periodic in $\theta$. For $(\theta_\delta,\rho)\in\mathcal N_{r_*}$ and $|\omega|<r_*/2$, the angular substitution \eqref{eq:subtraction-2bdtheta} gives
\begin{align*}
\theta=\theta_\delta+\psi(T(s))-\psi(T_\delta)+\omega,\qquad |\operatorname{Im}\theta|<3r_*/2,
\end{align*}
because the reference phase $\psi(T(s))-\psi(T_\delta)$ is real. Thus the preceding bound by $M$ applies uniformly after this substitution, regardless of the size of the real reference phase.

Set $C_0=10M$ and choose $R>C_0/6+1$. Decrease $\varepsilon_0$ if necessary so that the square-root branch in \eqref{energy-explicit-coordinates} satisfies
\begin{align}\label{eq:wv-v_0v_0}
|w|\le2\chi,\qquad |v-v_0|\le\tfrac12,\qquad v_0(\chi)=\sqrt{2-\chi^2}\ge1,
\end{align}
for $|\mathbb H|\le2R$, $0<\chi\le\delta\le1$ and $|\omega|<r_*/2$. 
Consequently, $|v|\ge1/2$ and $|v+v_0|\ge1$.

We first solve the phase equation for a prescribed function $\mathbb H$ with $\|\mathbb H\|_\infty<2R$, not yet required to satisfy \eqref{energy-Volterra}. Define
\begin{align*}
\mathcal P_{\mathbb H}(\omega)(\chi) =-\int_\chi^\delta B(s,\mathbb H(s),\omega(s);\theta_\delta,\rho)\,ds.
\end{align*}
The expression for $B$ in \eqref{energy-B} can equivalently be written as
\begin{align}\label{energy-B-difference}
B(s)=\frac{2\varepsilon^{-3}}{s^4}\left(\frac1v-\frac1{v_0}\right) -\frac2s\left(\frac wv-\frac{s}{v_0}\right).
\end{align}
Here $w,v$ are reconstructed from $(\mathbb H(s),\omega(s))$ by \eqref{energy-explicit-coordinates}. Using \eqref{energy-explicit-coordinates}, \eqref{eq:wv-v_0v_0} together with 
\begin{align*}
\frac wv-\frac{s}{v_0} =\frac{w-s}{v}+s\left(\frac1v-\frac1{v_0}\right), \qquad v^2-v_0^2=2\varepsilon^4s^4\mathbb H-(w-s)(w+s),
\end{align*}
and $w-s=\varepsilon^7s^7(\mathbb H-\mathcal V_4)$, we obtain
\begin{align*}
\|\mathcal P_{\mathbb H}(\omega)\|_\infty &\le\int_0^\delta \left[ \frac{2\varepsilon^{-3}}{s^4}\left|\frac1v-\frac1{v_0}\right| +\frac2s\left|\frac wv-\frac{s}{v_0}\right| \right]\,ds\\
&\le C\int_0^\delta \left[ \frac{\varepsilon^{-3}}{s^4} \left|2\varepsilon^4s^4\mathbb H-(w-s)(w+s)\right| +\frac{|w-s|}{s} \right]\,ds\\
&\le C_R\int_0^\delta \left[ \varepsilon^{-3}s^{-4}(\varepsilon^4s^4+\varepsilon^7s^8) +\varepsilon^7s^6 \right]\,ds \le C_R\varepsilon\delta.
\end{align*}
In particular, these integrals converge absolutely at $s=0$.

For comparisons, let $\theta_i,w_i,v_i$ be reconstructed from a pair $(\mathbb H_i,\omega_i)$ by \eqref{eq:subtraction-2bdtheta} and \eqref{energy-explicit-coordinates}, and write $Q_i(s)=Q(s,\theta_i(s),\rho)$ for $Q=\mathcal V_4,F_4,G_2$. Since $\theta_1-\theta_2=\omega_1-\omega_2$ by \eqref{eq:subtraction-2bdtheta}, integration of $\partial_\theta Q$ along the straight segment between the two angles and \eqref{eq:QpartialQ} give
\begin{align}\label{eq:Q1-Q2}
|Q_1(s)-Q_2(s)|\le M|\omega_1(s)-\omega_2(s)|.
\end{align}
The same algebraic reconstruction gives the exact identities from \eqref{energy-explicit-coordinates}, 
\begin{align}\label{energy-coordinate-differences}
\begin{split}
w_1-w_2 &=\varepsilon^7s^7 \left(\mathbb H_1-\mathbb H_2-\mathcal V_{4,1}+\mathcal V_{4,2}\right),\\
v_1-v_2 &=\frac{2\varepsilon^4s^4(\mathbb H_1-\mathbb H_2) -(w_1-w_2)(w_1+w_2)}{v_1+v_2}.
\end{split}
\end{align}
For the phase contraction, take $\mathbb H_1=\mathbb H_2=\mathbb H$. Equations \eqref{energy-B-difference}, \eqref{eq:Q1-Q2}, \eqref{energy-coordinate-differences} and $|v_1+v_2|\geq 1$ from \eqref{eq:wv-v_0v_0} yield
\begin{align*}
\|\mathcal P_{\mathbb H}(\omega_1)-\mathcal P_{\mathbb H}(\omega_2)\|_\infty &\le\int_0^\delta \left[ \frac{2\varepsilon^{-3}}{s^4}\left|\frac1{v_1}-\frac1{v_2}\right| +\frac2s\left|\frac{w_1}{v_1}-\frac{w_2}{v_2}\right| \right]\,ds\\
&\quad\le C\int_0^\delta \left[\varepsilon^{-3}s^{-4}|v_1-v_2|+s^{-1}|w_1-w_2|\right]\,ds\\
&\quad\le C\int_0^\delta (\varepsilon^{-3}s^{-3}+s^{-1})|w_1-w_2|\,ds\\
&\quad\le C\int_0^\delta (\varepsilon^4s^4+\varepsilon^7s^6) |\mathcal V_{4,1}-\mathcal V_{4,2}|\,ds \le C_R\varepsilon^4\delta^5\|\omega_1-\omega_2\|_\infty.
\end{align*}
Decrease $\varepsilon_0$ so that $C_R\varepsilon\delta<r_*/4$ and the preceding contraction constant is at most $1/2$. Then $\mathcal P_{\mathbb H}$ is a contraction of the closed ball $\|\omega\|_\infty\le C_R\varepsilon\delta$ into itself, yielding a unique fixed point denoted by $\omega[\mathbb H]$.

We next compare the phase solutions associated with two prescribed energy functions. Set $\omega_i=\omega[\mathbb H_i]$. Using their fixed-point equations and retaining the energy difference in \eqref{energy-coordinate-differences}, we obtain
\begin{align*}
\|\omega_1-\omega_2\|_\infty &\le C\int_0^\delta \left[\varepsilon^{-3}s^{-4}|v_1-v_2|+s^{-1}|w_1-w_2|\right]\,ds\\
&\le C\int_0^\delta \left[ \varepsilon|\mathbb H_1-\mathbb H_2| +(\varepsilon^{-3}s^{-3}+s^{-1})|w_1-w_2| \right]\,ds\\
&\le C_R\int_0^\delta \left[ \varepsilon|\mathbb H_1-\mathbb H_2| +\varepsilon^4s^4|\omega_1-\omega_2| \right]\,ds\\
&\le C_R\varepsilon\delta\|\mathbb H_1-\mathbb H_2\|_\infty +C_R\varepsilon^4\delta^5\|\omega_1-\omega_2\|_\infty.
\end{align*}
Decreasing $\varepsilon_0$ and absorbing the last term gives
\begin{align}\label{energy-phase-energy-dependence}
\|\omega[\mathbb H_1]-\omega[\mathbb H_2]\|_\infty \le C_R\varepsilon\delta\|\mathbb H_1-\mathbb H_2\|_\infty.
\end{align}

We now impose the energy equation by defining, for $\|\mathbb H\|_\infty\le R$,
\begin{align*}
\mathcal S(\mathbb H)(\chi) =\chi^{-6}\int_0^\chi s^5A(s,\mathbb H(s),\omega[\mathbb H](s);\theta_\delta,\rho)\,ds.
\end{align*}
Using $A=-2(F_4+(w/v)G_2)$ from \eqref{energy-B}, the gravitational coefficient bound \eqref{eq:QpartialQ} by $M$, and $|w/v|\le4s$ by \eqref{eq:wv-v_0v_0}, we have
\begin{align*}
\|\mathcal S(\mathbb H)\|_\infty \le2\sup_{0<\chi\le\delta}\chi^{-6} \int_0^\chi s^5\left(|F_4|+\left|\frac wv\right||G_2|\right)\,ds \le\frac{10M}{6}=\frac{C_0}{6}<R.
\end{align*}
Thus $\mathcal S$ maps the energy ball into itself.

For its contraction estimate, again set $\omega_i=\omega[\mathbb H_i]$ and use the comparison notation above. The formula for $A$, the bound \eqref{eq:QpartialQ} and \eqref{energy-coordinate-differences} give
\begin{align*}
\|\mathcal S(\mathbb H_1)&-\mathcal S(\mathbb H_2)\|_\infty \le C\sup_{0<\chi\le\delta}\chi^{-6} \int_0^\chi s^5 \left[ |\omega_1-\omega_2|+|w_1-w_2|+s|v_1-v_2| \right]\,ds\\
&\le C_R\sup_{0<\chi\le\delta}\chi^{-6} \int_0^\chi s^5 \left[ (\varepsilon^4s^5+\varepsilon^7s^7)|\mathbb H_1-\mathbb H_2| +(1+\varepsilon^7s^7)|\omega_1-\omega_2| \right]\,ds\\
&\le C_R\sup_{0<\chi\le\delta}\chi^{-6} \int_0^\chi s^5(\varepsilon^4s^5+\varepsilon\delta)\,ds\, \|\mathbb H_1-\mathbb H_2\|_\infty\\
&\le C_R(\varepsilon^4\delta^5+\varepsilon\delta) \|\mathbb H_1-\mathbb H_2\|_\infty,
\end{align*}
where the penultimate inequality uses \eqref{energy-phase-energy-dependence}. Since $\delta\le1$, $\mathcal S$ is a strict contraction for sufficiently small $\varepsilon_0$. Its fixed point $\mathbb H$, together with $\omega[\mathbb H]$, solves \eqref{energy-Volterra}--\eqref{energy-phase-Volterra}.  

It remains to prove holomorphic dependence on $(\theta_\delta,\rho)$.
The formulas \eqref{energy-explicit-coordinates},
\eqref{eq:subtraction-2bdtheta}, and \eqref{energy-B},
together with the uniform bounds established above, show that
substitution into $A,B$ is jointly holomorphic on the relevant
open function and parameter domains. The phase
Picard iterates starting at zero are then jointly holomorphic in
$(\mathbb H,\theta_\delta,\rho)$. The uniform contraction estimate
gives  uniform convergence in the Banach norm, so their
limit $\omega[\mathbb H]$ is jointly holomorphic by the
Banach-valued Weierstrass theorem. Applying the same argument
to the energy Picard iteration proves joint holomorphy of
$(\mathbb H,\omega)$ in $(\theta_\delta,\rho)$. The $2\pi$-periodicity in $\theta_\delta$ and reality for real
parameters follow from uniqueness and the corresponding
properties of the gravitational coefficients and the chosen
square-root branch. At $\rho=0$, the identities
$F_4=G_2=\mathcal V_4=0$ show that
$(\mathbb H,\omega)=(0,0)$ solves the Volterra system,
so it is the unique fixed point.

Finally, uniformly for $(\theta_\delta,\rho)\in\mathcal N_{r_*}$,
\eqref{energy-explicit-coordinates} gives
\begin{align*}
|w-\chi| &=\varepsilon^7\chi^7|\mathbb H-\mathcal V_4| \le C\varepsilon^7\chi^7,\qquad |v-v_0| =\frac{|2\varepsilon^4\chi^4\mathbb H-(w-\chi)(w+\chi)|}{|v+v_0|} \le C\varepsilon^4\chi^4.
\end{align*}
As $\chi,v_0$ are independent of $\rho$,  Cauchy's estimates on parameter discs of radius $r_*/2$ centred at real values $\theta_\delta\in\mathbb R$ and $\rho\in[0,1/2]$ yield
\begin{align*}
|\partial_\rho w|\le C\varepsilon^7\chi^7,\qquad |\partial_\rho v|\le C\varepsilon^4\chi^4.
\end{align*}
Applying the same argument to the bound $\|\omega\|_\infty\le C\varepsilon\delta$ gives $\|\partial_{\theta_\delta}\omega\|_\infty\le C\varepsilon\delta$. 
\end{proof}

The fixed point is a pair of functions on the entire radial half-interval,
not a fixed point of the finite-dimensional flow.  The following result
makes explicit both the reconstruction of the orbit and its invariance.

\begin{proposition}
\label{prop:energy-local-invariant}
For fixed real $(\varepsilon,\rho)\in(0,\varepsilon_0)\times[0,\frac12]$, reconstruct $(v,w,\theta)$ from $(\mathbb H,\omega)$ in Lemma~\ref{lem:energy-local-graph}  using \eqref{energy-explicit-coordinates} and \eqref{eq:subtraction-2bdtheta}. Then the resulting map $(\chi,\theta_\delta)\mapsto(\chi,v,w,\theta)$ forms a two-dimensional immersed surface $W^s_{\mathrm{loc},\delta}$ in the Jacobi manifold $J=-\varepsilon^{-1}$ on $0<\chi\le\delta$, positively invariant under the original flow. Every trajectory on this surface exists for all future physical times, with $\chi(t)\downarrow0$ and $(X(t),Y(t))\to(0,0)$ as $t\to\infty$. At each fixed $\chi\in(0,\delta]$, its radial section is a real-analytic circle graph over $\theta$, whose coordinates depend jointly real analytically on $(\theta,\rho)$.
\end{proposition}

\begin{proof}
It follows from the construction in \Cref{lem:energy-local-graph}, particularly \eqref{eq:wv-v_0v_0} that $1/2\le v\le2$ for real parameters, so the rescaled physical time $\sigma=\varepsilon^3t$ is recovered by 
\begin{align*}
\sigma(\chi)=2\int_\chi^\delta\frac{ds}{s^4v(s)}, \quad \text{ as } \frac{d\chi}{d\sigma}=-\tfrac12\chi^4v.
\end{align*}
It is strictly decreasing as a function of $\chi$, with
$\sigma(\delta)=0$ and $\sigma(\chi)\to+\infty$ as $\chi\to0$.
Hence it has a unique inverse $\chi(\sigma)$ for every
$\sigma\ge0$.  By Proposition~\ref{prop:energy-reduction}, the
reconstructed curve, composed with this inverse, solves the original
system on the selected Jacobi manifolds.  The formulas in
\eqref{energy-explicit-coordinates} give $w=O(\chi)$,
$v\to\sqrt2$, and therefore $(X,Y)\to(0,0)$ as $X=w$ and $Y=-vw/\sqrt2$.

For fixed $\chi\in(0,\delta]$ the phase map is
\begin{align*}
\theta_\delta\longmapsto \theta_\delta+\psi(T(\chi))-\psi(T_\delta) +\omega(\chi;\theta_\delta,\rho).
\end{align*}
Since 
$\partial_{\theta_\delta}\omega=O(\varepsilon\delta)$ by \eqref{energy-graph-bounds}, this map has
derivative $1+O(\varepsilon\delta)$ and is a degree-one
real-analytic circle diffeomorphism.  Together with the independent
radial coordinate $\chi$, it makes the parametrization
\begin{align*}
(\chi,\theta_\delta)\longmapsto \bigl(\chi,v(\chi;\theta_\delta),w(\chi;\theta_\delta), \theta_\delta+\psi(T(\chi))-\psi(T_\delta) +\omega(\chi;\theta_\delta)\bigr)
\end{align*}
an immersion for $0<\chi\le\delta$. The invariance follows from $\frac{d\chi}{d\sigma}<0$ on the surface.  For each fixed $\chi\in(0,\delta]$, Lemma~\ref{lem:energy-local-graph} and the reconstruction formulas \eqref{energy-explicit-coordinates}--\eqref{eq:subtraction-2bdtheta} give joint real-analytic dependence on $(\theta_\delta,\rho)$ and hence on $(\theta,\rho)$ by the analytic inverse function theorem as the phase map has nonvanishing $\theta_\delta$-derivative. 
\end{proof}

\subsection{Analytic transition and identification of all primary branches}
The local contraction selects a forward-complete stable surface near
infinity.  We first continue this surface by a finite time in the
regular Duffing chart to obtain the analytic section traces required
by the theorem.  We then verify that every primary stable tail belongs
to the same local surface, which upgrades uniqueness inside the
fixed-point ball to uniqueness of primary branches.

Recall that $0<\delta\le1$, and $T(\chi)$ is the inverse of the function $\chi=a(\tau)$ as in \eqref{eq:subtraction-2bdtheta}, with $T_\delta=T(\delta)$ such that $a(T_\delta)=\delta$. Besides, $\theta_0$ denotes the actual phase in \eqref{e.050407}, and as given after \eqref{e.050402}, $I$ denotes a small segment of the $X$-axis centered at $a(0) = \sqrt{2}$.  We denote 
\begin{align*}
\Sigma=\{Y=0,\ X\in I\}.
\end{align*}
\begin{lemma}
\label{lem:analytic-transition}
Fix $0<\delta\le1$. After decreasing $\varepsilon_0$ and $r_*$
if necessary, for each $\varepsilon\in(0,\varepsilon_0)$ and
$\rho\in[0,\frac12]$, every trajectory of the local stable surface
in Proposition~\ref{prop:energy-local-invariant} has a first
backward intersection with $\Sigma$.
The intersection coordinates
$X^s_\Sigma(\theta_\delta,\rho;\varepsilon)$ and
$\theta^s_\Sigma(\theta_\delta,\rho;\varepsilon)$ extend jointly
holomorphically to $(\theta_\delta,\rho)\in\mathcal N_{r_*}$,
and satisfy
\begin{align}\label{phase-transition-close}
\theta^s_\Sigma(\theta_\delta,\rho;\varepsilon) =\theta_\delta-\psi(T_\delta)+O_\delta(\varepsilon),\quad \partial_{\theta_\delta}\theta^s_\Sigma =1+O_\delta(\varepsilon).
\end{align}
Consequently, these intersections form a graph over the actual
section phase $\theta_0$, with joint real-analytic dependence on
$(\theta_0,\rho)$. Normalizing Duffing time by $\tau=0$ at this
intersection, the real stable branches satisfy
\begin{align}\label{global-branch-closeness}
|X^s(\tau)-a(\tau)|+|Y^s(\tau)-b(\tau)| \le C_\delta\varepsilon^4a(\tau),\qquad \tau\ge0.
\end{align}
The analogous assertions hold for the unstable branches by
reversibility.
\end{lemma}
\begin{proof}
Fix $\delta$ and decrease $r_*$ so that the local graph in
Lemma~\ref{lem:energy-local-graph} is holomorphic on
$\mathcal N_{4r_*}$. We obtain estimates on $\mathcal N_{2r_*}$
and use Cauchy's estimates on $\mathcal N_{r_*}$.

Let $(X(\tau),Y(\tau),\theta(\tau))$ solve \eqref{EqnXYU1}
with initial data supplied by the local graph in
Lemma~\ref{lem:energy-local-graph} at $\chi=\delta$:
\begin{align*}
(X(T_\delta),Y(T_\delta),\theta(T_\delta)) =\left(w(\delta),-\frac{v(\delta)w(\delta)}{\sqrt2}, \theta_\delta\right).
\end{align*}
By \eqref{e.050303} and the uniform complex estimates for
$w-\chi$ and $v-v_0$ obtained at the end of the proof of
Lemma~\ref{lem:energy-local-graph}, we have, uniformly for
$(\theta_\delta,\rho)\in\mathcal N_{2r_*}$,
\begin{align}\label{eq:XYTdelta}
X(T_\delta)=\delta+O(\varepsilon^7\delta^7),\qquad Y(T_\delta)=b(T_\delta)+O(\varepsilon^4\delta^5).
\end{align}
Write $Z=(X,Y)$, $Z_0=(a,b)$ and
$\theta_{\mathrm{ref}}(\tau)=\theta_\delta+\psi(\tau)-\psi(T_\delta)$.
Choose a small $\eta>0$ and a bounded complex neighborhood
$\mathcal D$ of $\{Z_0(\tau):-\eta\le\tau\le T_\delta\}$
on which $|X|\ge c_\delta>0$.
For $Z\in\mathcal D$ and $(\theta,\rho)\in\mathcal N_{4r_*}$,
the Jacobi equation \eqref{Jacobi1} defines a holomorphic branch
$\mathcal{J}=1+O_\delta(\varepsilon^3)$. Recalling $\mathcal V=(1-\rho)R_{13}^{-1}+\rho R_{23}^{-1}$, the barycentric cancellations in
\eqref{e.050302} and \eqref{e.050403} give
\begin{align}\label{eq:VFG}
|\mathcal V-1|+|F|\le C_\delta\varepsilon^4,\qquad |G|\le C_\delta\varepsilon^2,
\end{align}
with the same bounds for first derivatives with respect to
$(X,Y,\theta,\rho)$. As long as $Z\in\mathcal D$ and
$|\theta-\theta_{\mathrm{ref}}|<r_*$, equation
\eqref{EqnXYU1}, \eqref{eq:VFG} and \eqref{eq:XYTdelta} give for $-\eta\le\tau\le T_\delta$, 
\begin{align*}
|Z(\tau)-Z_0(\tau)| &\le C_\delta\varepsilon^4 +C_\delta\int_\tau^{T_\delta} \bigl(|Z(s)-Z_0(s)|+|F|+\varepsilon^2|G|\bigr)\,ds\\
&\le C_\delta\varepsilon^4 +C_\delta\int_\tau^{T_\delta}|Z(s)-Z_0(s)|\,ds.
\end{align*}
Thus Gronwall's inequality yields
\begin{align}\label{Z-Z0}
\sup_{-\eta\le\tau\le T_\delta}|Z(\tau)-Z_0(\tau)|\le C_\delta\varepsilon^4.
\end{align}
Since $a^2-b^2-a^4/2=0$, \eqref{Jacobi1} and \eqref{eq:VFG} then imply 
\begin{align*}
|\mathcal{J}-1| \le C_\delta\varepsilon^3 \left(\left|X^2-Y^2-\frac12X^4\right|+|\mathcal V-1|\right)\le C_\delta\varepsilon^3 \bigl(|Z-Z_0|+\varepsilon^4\bigr) \le C_\delta\varepsilon^7.
\end{align*}
Subtracting the reference angular equation from the equation of $\theta$ in \eqref{EqnXYU1} consequently gives
\begin{align*}
|\theta(\tau)-\theta_{\mathrm{ref}}(\tau)| &\le\sqrt2\int_\tau^{T_\delta} \left(|X-a|+\varepsilon^{-3}|\mathcal{J}^{-3}X^{-3}-a^{-3}|\right)\,ds\\
&\le C_\delta\int_\tau^{T_\delta} \left((1+\varepsilon^{-3})|Z-Z_0| +\varepsilon^{-3}|\mathcal{J}-1|\right)\,ds \le C_\delta\varepsilon.
\end{align*}
For real $\tau$, $\psi(\tau)-\psi(T_\delta)$ is real. For sufficiently small $\varepsilon_0$, these estimates keep
$Z$ a positive distance from $\partial\mathcal D$ and give
$|\theta-\theta_{\mathrm{ref}}|\le C_\delta\varepsilon<r_*/2$.
The standard continuation argument therefore extends the
solution to $[-\eta,T_\delta]$, where all the preceding
estimates hold. The continued solution is jointly
holomorphic in $(\theta_\delta,\rho)\in\mathcal N_{2r_*}$
by holomorphic dependence for the ODE.

At $\tau=0$, we have
$X(0)=\sqrt2+O_\delta(\varepsilon^4)$ and
$Y(0)=O_\delta(\varepsilon^4)$.
By the holomorphic ODE theorem, the solution admits a local
extension jointly holomorphic in complex time $\tau$ and
the parameters $(\theta_\delta,\rho)$.
The vector field in \eqref{EqnXYU1} and its first derivatives
with respect to $(X,Y,\theta)$ are $O_\delta(\varepsilon^{-3})$.
Thus, for each fixed $K_\delta>0$, after decreasing $\varepsilon_0$,
this extension is defined on a neighborhood of
$|\tau|\le K_\delta\varepsilon^4$, uniformly for
$(\theta_\delta,\rho)\in\mathcal N_{2r_*}$.
On this disc, 
\begin{align*}
|Z(\tau)-Z(0)|\le C_\delta|\tau|,\qquad |\theta(\tau)-\theta(0)|\le C_\delta\varepsilon^{-3}|\tau|, \qquad Y'(\tau)=-\sqrt2+O_\delta(\varepsilon^4),
\end{align*}
by \eqref{EqnXYU1} and  \eqref{eq:VFG}. 
For sufficiently large $K_\delta$, Rouch\'e's theorem and the
holomorphic implicit function theorem therefore give a unique
zero $\tau_*=O_\delta(\varepsilon^4)$ of $Y$, jointly holomorphic
in $(\theta_\delta,\rho)$.
For real parameters, uniqueness makes $\tau_*$ real.
Equations \eqref{EqnXYU1}, \eqref{eq:VFG} and
\eqref{Z-Z0} give
$Y'=b'+O_\delta(\varepsilon^4)$ on $[-\eta,T_\delta]$.
Choosing $\eta>0$ sufficiently small and then decreasing
$\varepsilon_0$, we have $Y'<0$ on $[-\eta,\eta]$
and $Y<0$ on $[\eta,T_\delta]$.
On the local branch $\tau\ge T_\delta$,
\eqref{energy-graph-bounds} gives
$Y=-vw/\sqrt2<0$.
Since $X(\tau_*)=\sqrt2+O_\delta(\varepsilon^4)\in I$,
$\tau_*$ is the first backward intersection with $\Sigma$. Evaluation at $\tau_*$ gives
$X^s_\Sigma=\sqrt2+O_\delta(\varepsilon^4)$ and, since $\psi(0)=0$,
\begin{align*}
|\theta^s_\Sigma-\theta_\delta+\psi(T_\delta)| &\le |\theta(0)-\theta_{\mathrm{ref}}(0)| +C_\delta\varepsilon^{-3}|\tau_*| \le C_\delta\varepsilon.
\end{align*}
Cauchy's estimate on $\mathcal N_{r_*}$ yields
$\partial_{\theta_\delta}\theta^s_\Sigma=1+O_\delta(\varepsilon)$.
Periodicity and this derivative estimate make the real phase map
a degree-one circle diffeomorphism. The analytic inverse function
theorem then gives joint real-analytic dependence of the section
graph on its actual phase $\theta_0$ and $\rho$.

For real parameters, set
$(X^s(\tau),Y^s(\tau))=(X(\tau+\tau_*),Y(\tau+\tau_*))$,
so that the intersection occurs at $\tau=0$. On the branch $0<\chi\le\delta$ constructed in
Proposition~\ref{prop:energy-local-invariant},
$d\chi/d\tau=-\chi v/\sqrt2$.
After the time shift, $\tau(\delta)=T_\delta-\tau_*$,
and hence
\begin{align*}
\tau(\chi)=T_\delta-\tau_*+\sqrt2\int_\chi^\delta\frac{dr}{r v(r)}.
\end{align*}
Since $T(\chi)=T_\delta+\sqrt2\int_\chi^\delta dr/(r v_0(r))$ as it is the inverse of $a$,
\eqref{energy-graph-bounds} gives
\begin{align*}
|\tau(\chi)-T(\chi)| \le|\tau_*|+\sqrt2\int_\chi^\delta \frac{|v(r)-v_0(r)|}{r v(r)v_0(r)}\,dr\le C_\delta\varepsilon^4 +C\varepsilon^4\int_\chi^\delta r^3\,dr \le C_\delta\varepsilon^4.
\end{align*}
Since $X^s(\tau(\chi))=w(\chi)$ and  $Y^s(\tau(\chi))=-\frac{v(\chi) w(\chi)}{\sqrt{2}}$, the bounds in \eqref{energy-graph-bounds} imply
\begin{align*}
|X^s(\tau(\chi))-\chi| +|Y^s(\tau(\chi))-b(T(\chi))| \le C|w-\chi|+C\chi|v-v_0|\le C\varepsilon^4\chi^5.
\end{align*}
The Duffing functions satisfy $|a'|\le a$ and $|b'|\le Ca$.
Writing $\tau=\tau(\chi)$ and using $a(T(\chi))=\chi$, we obtain $\chi\le e^{|\tau-T(\chi)|}a(\tau)$ and 
\begin{align*}
&|X^s(\tau)-a(\tau)|+|Y^s(\tau)-b(\tau)|\\
&\le \left|X^s(\tau)-\chi\right|+\left|Y^s(\tau)-b(T(\chi))\right| +|\chi-a(\tau)|+|b(T(\chi))-b(\tau)|\\
&\le C\varepsilon^4\chi^5 +C|\tau-T(\chi)|e^{|\tau-T(\chi)|}a(\tau)\le C_\delta\varepsilon^4a(\tau).
\end{align*}
On the remaining interval $0\le\tau\le T_\delta-\tau_*$, the same conclusion follows from
$|Z-Z_0|\le C_\delta\varepsilon^4$ by \eqref{Z-Z0}, the time shift
$|\tau_*|\le C_\delta\varepsilon^4$, and the positive lower
bound for $a$. This proves \eqref{global-branch-closeness}.
The unstable assertions follow from the reversibility
$(X,Y,\theta,\tau)\mapsto(X,-Y,-\theta,-\tau)$.
\end{proof}

We next show that every primary stable solution is captured by
the local fixed-point construction near $\chi=0$.

\begin{lemma}
\label{lem:primary-asymptotics}
For every primary stable solution $(X,Y)$, there exists $0<\delta_*\le1$
such that its restriction to $0<\chi\le\delta_*$ coincides with
the trajectory reconstructed from the fixed point in
Lemma~\ref{lem:energy-local-graph}, with $\delta=\delta_*$
and $\theta_{\delta_*}$ equal to its actual phase at $\chi=\delta_*$.
\end{lemma}

\begin{proof}
We show that an arbitrary primary stable solution satisfies
\eqref{energy-Volterra}--\eqref{energy-phase-Volterra} in the
uniqueness ball of Lemma~\ref{lem:energy-local-graph},
after restricting it to a sufficiently small radial interval. We first reparametrize the primary stable solution with  radial variable.  By Definition~\ref{def:primary}, $X(\tau)>0$ at every finite time
and $(X(\tau),Y(\tau))\to(0,0)$ as $\tau\to+\infty$.
Near $(X,Y)=(0,0)$, the Jacobi equation \eqref{Jacobi1} and
the gravitational formulas \eqref{e.050302}--\eqref{e.050403} give
\begin{align*}
\mathcal{J}=1+O(X^2+Y^2),\qquad F=O(\varepsilon^4X^4),\qquad G=O(\varepsilon^2X^2),
\end{align*}
uniformly for real angles and parameters.
Thus, by \eqref{EqnXYU1}, we may choose $\tau_1$ sufficiently large
that, for all $\tau\ge\tau_1$,
\begin{align*}
Y'=X\left(1-X^2-F+\sqrt2\varepsilon^2\mathcal{J}^2YG\right) \ge\tfrac12X>0,\qquad X^2+|F|\le\tfrac14,
\end{align*}
where primes denote differentiation in Duffing time $\tau$.
Since $Y(\tau)\to0$, it follows that $Y(\tau)<0$ and hence
$v=-\sqrt2Y/X>0$ on this interval. Differentiating this expression for $v$ and using
\eqref{EqnXYU1}, the terms involving $G$ cancel, giving
\begin{align*}
v'=-\sqrt2\left(\frac{Y'}X-\frac{YX'}{X^2}\right) =\frac{v^2+2X^2-2+2F}{\sqrt2} \le-\frac1{2\sqrt2} \qquad\text{whenever }0<v\le1.
\end{align*}
If $v(\tau_0)\le1$ at some $\tau_0\ge\tau_1$, this inequality
prevents an upward crossing of $v=1$ and forces $v$ to reach zero
in finite time, contradicting $v>0$. Therefore $v>1$ for
$\tau\ge\tau_1$.

Since $\chi=\mathcal{J}X\to0$ and $w/\chi=1/\mathcal{J}\to1$, after increasing
$\tau_1$ if necessary we may set $\delta_*=\chi(\tau_1)\le1$.
The identity $d\chi/d\tau=-\chi v/\sqrt2<0$ then allows us to
parametrize the solution by $\chi$ on $(0,\delta_*]$, where $ v\ge1,\, 0<w\le2\chi$. 
Moreover, $\chi=\mathcal{J}X$ and \eqref{XY} give the boundary condition
\begin{align*}
\chi^2\left(\tfrac12(v^2+w^2)-1\right) =\mathcal{J}^2\left(Y^2+\tfrac12X^4-X^2\right)\longrightarrow0 \qquad(\chi\to0).
\end{align*}
Using the exact energy identity \eqref{eq:energy-chi}, 
we may therefore integrate from zero to obtain
\begin{align*}
\mathbb H(\chi) =-2\chi^{-6}\int_0^\chi s^5 \left(F_4(s,\theta(s),\rho) +\frac{w(s)}{v(s)}G_2(s,\theta(s),\rho)\right)\,ds.
\end{align*}
Let $M$ be the coefficient bound in \eqref{eq:QpartialQ}.
Since $w(s)/v(s)\le2s$ and $\delta_*\le1$,
\begin{align*}
|\mathbb H(\chi)| \le2M\chi^{-6}\int_0^\chi s^5(1+2s)\,ds \le M<R,
\end{align*}
where $R$ is the radius chosen in the proof of
Lemma~\ref{lem:energy-local-graph}. The reconstruction formulas \eqref{energy-explicit-coordinates}
now imply $|w-\chi|\le C\varepsilon^7\chi^7$ and 
\begin{align*}
|v-v_0| =\frac{|2\varepsilon^4\chi^4\mathbb H-(w-\chi)(w+\chi)|}{v+v_0} \le C\varepsilon^4\chi^4.
\end{align*}
Set $\theta_{\delta_*}=\theta(\delta_*)$ and define $\omega$
by \eqref{eq:subtraction-2bdtheta} with $\delta=\delta_*$.
Then $\omega(\delta_*)=0$, and \eqref{energy-B-difference} yields
\begin{align*}
\left|\frac{d\omega}{d\chi}\right| &\le C\left( \varepsilon^{-3}\chi^{-4}|v-v_0| +\frac{|w-\chi|}{\chi}+|v-v_0| \right) \le C\varepsilon,\\
\|\omega\|_\infty &\le\int_0^{\delta_*}|B(s)|\,ds \le C\varepsilon\delta_*.
\end{align*}
This is the same phase estimate used in the proof of
Lemma~\ref{lem:energy-local-graph}, with its chosen constant $C$.
The energy integral above and the phase equation with
$\omega(\delta_*)=0$ show that $(\mathbb H,\omega)$ satisfies
\eqref{energy-Volterra}--\eqref{energy-phase-Volterra}
in that lemma's uniqueness ball.
It therefore coincides with the fixed point having crossing
phase $\theta_{\delta_*}$, and the reconstruction formulas
give the asserted identification of trajectories.
\end{proof}

We are now ready to give the proof of the main result in this section. 
\begin{proof}[Proof of Proposition~\ref{props2.1}]
Decrease $\varepsilon_0$ so that the mass-ratio interval in
\eqref{e.051102} lies in the common real parameter neighborhood
of the preceding results, and fix $\delta=1$.
Lemma~\ref{lem:energy-local-graph} and
Proposition~\ref{prop:energy-local-invariant} give exact stable
trajectories parametrized by their phase $\theta_1$ at $\chi=1$.
Lemma~\ref{lem:analytic-transition} continues them to $\Sigma$
and gives an analytic section graph over $(\theta_0,\rho)$.
After setting $\tau=0$ at this intersection,
\eqref{global-branch-closeness} ensures, for sufficiently small
$\varepsilon_0$, that these trajectories remain in the  tube neighborhood 
$\mathcal D_\ell^+$ introduced after \eqref{e.050402}, satisfy $X^s(\tau)\ge a(\tau)/2>0$,
and approach $(0,0)$. Thus they satisfy Definition~\ref{def:primary}.
By Lemma~\ref{lem:primary-asymptotics}, every primary stable
solution agrees near $\chi=0$ with a locally constructed solution.
The radial section graphs in
Proposition~\ref{prop:energy-local-invariant}, together with
local fixed-point uniqueness after phase normalization at a common radial section, identify this solution with a member
of the family constructed at $\delta=1$.
ODE uniqueness then identifies their continuations.
For the fixed sufficiently small tube $\mathcal D_\ell^+$
and interval $I\subset(1,\infty)$, equation \eqref{EqnXYU1} gives
\begin{align*}
X'<0\quad\text{on }\mathcal D_\ell^+\cap\{X=1\},\qquad Y'<0\quad\text{on }\mathcal D_\ell^+\cap\{X\ge1,\ Y=0\}.
\end{align*}
A trajectory remaining in the tube cannot cross $X=1$
from left to right and, while $X\ge1$, can cross $Y=0$
only from above to below. Hence it meets $\Sigma$ at most once,
so their initial section points coincide.
The section-phase diffeomorphism in
\eqref{phase-transition-close} therefore proves uniqueness
for each prescribed $\theta_0$.
Reversibility gives the corresponding unstable assertions.
\end{proof}

\section{Existence of transversal homoclinic intersections}\label{s3}
In Section \ref{s2}, we treated $\varepsilon$ as the perturbation parameter, placing no restriction on $\rho$. This perspective was effective for constructing primary stable and unstable solutions. However, the splitting distance between the invariant manifolds is {\bf not} real analytic in $\varepsilon$ at $\varepsilon = 0$, which prevents us from expanding it as a power series in $\varepsilon$ for approximation purposes. To address this, in this section we instead assume
$$
0<\rho \ll \varepsilon \ll 1
$$
and regard $\rho$ as the perturbation parameter. The main result of this section is 
\begin{proposition}\label{prop3.1}
There exists a constant $\varepsilon_1 > 0$, with $\varepsilon_1 \ll \varepsilon_0$ as in Proposition \ref{props2.1}, such that for every fixed $\varepsilon \in (0, \varepsilon_1)$, there exists a corresponding $\rho(\varepsilon) > 0$ sufficiently small so that, for all $\rho \in (0, \rho(\varepsilon))$, the primary stable and unstable manifolds admit transversal homoclinic intersections.
\end{proposition}

The contents of this section are as follows.  In Subsection \ref{s2.1}, following \cite{CW}, we derive a useful integral equation of the primary stable/unstable solutions.  In subsection \ref{s3.1}, we derive a counterpart of the Poincar\'e-Melnikov integral for the restricted three-body problem. We then expand this integral into a Fourier series in $\theta_0$ in subsection \ref{s3.2},  and estimate the Fourier coefficients in subsection \ref{s3.3}. The proof of Proposition \ref{prop3.1} and  the main theorem is given at the end.

\subsection{Canonical coordinates and integral equations} \label{s2.1}
To derive integral equations for the primary stable solutions, following  \cite{CW},  we let  
\begin{align*}
h(\tau) = \frac{3(e^{2\tau} - e^{-2\tau} + 4\tau)}{2(e^\tau+ e^{-\tau})^2},
\end{align*}
and 
\begin{align*}
H(\tau) = \frac{1}{a(\tau)}\left[b(\tau)h(\tau) + a(\tau)\right], \ \ \ \ \ \ \ \widetilde{H}(\tau) = \frac{1}{a(\tau)}\left[b'(\tau)h(\tau) + 2b(\tau)\right].
\end{align*} 
We note that $h(\tau), H(\tau), \widetilde{H}(\tau)$ are uniformly bounded functions for all real $\tau$, $h(\tau)$ and $\widetilde{H}(\tau)$ are odd, but $H(\tau)$ is even in $\tau$. In addition, $h(\tau)$ solves 
\begin{align*}
h' - \frac{2b}{a} h - 3 = 0.
\end{align*}

Following the design of \cite{CW},   we let $M, W$ be such that 
\begin{equation}\label{xyMW}
M = \frac{1}{a}\left(b' x -b y\right), \ \ W = \left(\widetilde{H} x -H y\right).
\end{equation}
We have, in reverse, 
\begin{equation}\label{MWxy}
x = \frac{1}{a}\left(b W - aH M\right), \ \ y = \frac{1}{a}\left(b' W -a\widetilde{H} M\right).
\end{equation}
New variables $M, W$ are designed to transform  the equations of the first variations of the unperturbed Duffing equation, that is,
$$
\frac{d x}{d \tau} = y, \ \ \ \ \frac{dy}{d\tau} = (1-3a^2(\tau)) x,
$$
to 
$$
\frac{d M}{d \tau} = -\frac{b(\tau)}{a(\tau)}M, \ \ \ \ \frac{d W}{d\tau} = \frac{b(\tau)}{a(\tau)} W.
$$
\begin{lemma}\label{lem:canonical-MW}
	Equation \eqref{eqnxyS} in  $M, W, \Theta$ is transformed to 
	\begin{align}\label{EqnMWTheta}
\begin{split}
\frac{d \Theta}{d\tau} =& \frac{\sqrt{2} S}{\varepsilon^3 \mathcal{J}_{\theta_0, \psi}^3 (x+a)^3 a^3},\\
\frac{dM}{d\tau} = & -\frac{b}{a}M + \frac{1}{a}\left(b' P - b Q\right),\qquad \frac{dW}{d\tau}= \frac{b}{a} W + \widetilde{H} P -HQ
\end{split}
\end{align}
where on the right-hand side, $S, P, Q$ are as in \eqref{SPQS1}, in which we need to further change from $x, y$ to $M, W$ by using \eqref{MWxy}.
\end{lemma}
\begin{proof} The equation on $\Theta$ is direct. For $M$ we have  
\begin{eqnarray*}
\frac{dM}{d\tau}  & = &  -\frac{b}{a} M + \frac{1}{a}\left(b' x'+b'' x  -b' y-by'\right)\\
& = &  -\frac{b}{a} M + \frac{1}{a}\left(b' (y+P)+b'' x  -b' y-b((1-3a^2) x + Q)\right)\\
& = &  -\frac{b}{a} M + \frac{1}{a}\left(b' P -b  Q\right)
\end{eqnarray*}
where for the last equality we used $b'' = (a-a^3)' = (1-3a^2)b$.

Now for $W$, we have
\begin{eqnarray*}
  \frac{dW}{d\tau}   &=&  \left(\frac{1}{a }\left[b' h  + 2b \right]\right)' x -\left( \frac{1}{a }\left[b h  + a \right]\right)' y + \left(\frac{1}{a }\left[b' h  + 2b \right]\right) x' -\left( \frac{1}{a }\left[b h  + a \right]\right)  y' \\ 
    &=& -\frac{b}{a} W + \frac{1}{a }\left[b''h +b'h'+ 2b'\right] x - \frac{1}{a}\left[b'h+ bh'+ b\right] y  +   \frac{1}{a}\left[b'h + 2b\right]( y   + P )\\
    & &-\frac{1}{a}\left[bh + a\right] ((1-3a^2) x + Q).  
  \end{eqnarray*}
To continue, we use $h' = \frac{2b}{a} h +3$
to obtain
\begin{eqnarray*}
	\frac{dW}{d\tau}   
	&=& -\frac{b}{a} W + \frac{1}{a }\left[b'' h  +b' \left(\frac{2b}{a} h + 3\right)+ 2b' \right] x  - \frac{1}{a }\left[b' h + b  \left(\frac{2b}{a} h + 3\right) + b \right] y \\
	& &   +   \frac{1}{a }\left[b' h  + 2b \right]y  +   \frac{1}{a }\left[b' h  + 2b \right]P -\frac{1}{a }\left[b h  + a \right] (1-3a^2) x  -\frac{1}{a }\left[b h  + a \right] Q\\
	 &=& \frac{b}{a} W   +  \widetilde{H} P -HQ.  
\end{eqnarray*}
Here,  we used
$$
b' = a-a^3, \ \ \ \ b'' = (1-3a^2) b, \ \ \ \ b^2 = a^2 - \frac{1}{2}a^4
$$
for the last equality. 

\end{proof}

\begin{lemma}\label{l.050402}
	The primary stable solutions as in \Cref{def:primary} are solutions of the integral equations
	\begin{align}\label{eqnint1}
\begin{split}
\Theta(\tau) =& \int_0^{\tau} \frac{\sqrt{2} S}{\varepsilon^3 \mathcal{J}_{\theta_0, \psi}^3 (x+a)^3 a^3} d\tau,\\
M(\tau) = & -\frac{1}{a} \int_{\tau}^{+\infty} \left(b' P - b Q\right) d\tau,\\
W(\tau)=& a \int_0^{\tau}\frac{1}{a}(\widetilde{H} P -HQ) d\tau.
\end{split}
\end{align}
 \end{lemma}
\begin{proof}
	With a fixed $\theta_0$, we are interested in the solution of equation \eqref{EqnMWTheta} satisfying
    $$
\Theta(0) = 0, \ \ \ M(0) = -(X_0-a(0)), \ \ \ W(0) = 0.
$$
	Consequently, the solution we are seeking satisfies the integral equations
	\begin{equation*}
\begin{split}
\Theta(\tau) =& \int_0^{\tau} \frac{\sqrt{2} S}{\varepsilon^3 \mathcal{J}_{\theta_0, \psi}^3 (x+a)^3 a^3} d\tau,\\
M(\tau) = & \frac{1}{a} \left(-\sqrt{2}(X_0-a(0)) + \int_0^{\tau} \left(b' P - b Q\right) d\tau\right),\\
W(\tau)=& a \int_0^{\tau}\frac{1}{a}(\widetilde{H} P -HQ) d\tau.
\end{split}
\end{equation*}
	For a primary stable solution, we have $\lim_{\tau \to +\infty} a(\tau) M(\tau) = 0$,
	which implies
	$$
\sqrt{2}(X_0-a(0)) = \int_0^{+\infty} \left(b' P - b Q\right) d\tau.
$$
	Substituting it to the integral equation for $M(\tau)$, we obtain 
	$$
M(\tau) = -\frac{1}{a} \int_{\tau}^{+\infty} \left(b' P - b Q\right) d\tau.
$$
\end{proof}

\begin{remark}\label{r.050401}
	By a similar calculation, one obtains 
	\begin{align*}
M(\tau) = \frac{1}{a} \int_{-\infty}^{\tau} \left(b' P - b Q\right) d\tau
\end{align*}
	for primary unstable solutions. 
\end{remark}

\subsection{Splitting distance and its first approximation} \label{s3.1}
To analyze the intersection behavior of the stable and unstable manifolds, we introduce the splitting distance, which quantifies the separation between these manifolds at a fixed section. For the normalization of the splitting distance, set
\begin{align*}
{\mathbb M} = \frac{1}{\varepsilon^3\sqrt{\varepsilon} a} M, \quad {\mathbb W} = \frac{1}{\varepsilon^3\sqrt{\varepsilon} a} W
\end{align*}
with $(M,W)$ defined in terms of $(x,y)$ from \eqref{xyMW}. 
Fixing $(\rho, \varepsilon) \in D_{\ee_0}$, the Jacobi integral with $J = - \varepsilon^{-1}$ defines a three-dimensional invariant surface in the original four-dimensional phase space for the restricted three-body problem. For a given phase $\theta_0 \in {\mathbb R}$, we denote the renormalized primary stable solution as 
$$
{\mathcal V}^s(\tau, \theta_0, \rho, \varepsilon) = ({\mathbb M}^s(\tau, \theta_0, \rho, \varepsilon), \ {\mathbb W}^s(\tau, \theta_0, \rho, \varepsilon), \Theta^s(\tau, \theta_0, \rho, \varepsilon)), \ \ \tau\in [0, +\infty),
$$ 
and the  renormalized primary unstable solution as 
$$
{\mathcal V}^u(\tau, \theta_0, \rho, \varepsilon) = ({\mathbb M}^u(\tau, \theta_0, \rho, \varepsilon), \ {\mathbb W}^u(\tau, \theta_0, \rho, \varepsilon),\Theta^u(\tau, \theta_0, \rho, \varepsilon)), \ \ \tau \in (-\infty, 0].
$$ 
The family of primary stable/unstable solutions parametrized by $\theta_0 \in \mathbb{R}$ forms an immersed two-dimensional manifold, denoted by $W^s$/$W^u$, within the Jacobi integral manifold.

We emphasize that the time variable $\tau$ is used only as an orbit-wise parametrization. Since
$$
\frac{d\tau}{dt}=\frac{1}{\sqrt{2}}u^{3/2}>0
$$
away from infinity, the original $t$-flow and the $\tau$-flow have the same oriented orbits. Thus the notation $\tau=0$ in the stable and unstable parametrizations does not refer to a prescribed common value of the physical time $t$. It denotes the intersection with the common geometric section specified by $Y=0$ and by the phase condition $\Theta=0$ from \eqref{e.050407} and \eqref{iniS}. Equivalently, both parametrizations are therefore compared at section points
with the same actual phase $\theta_0$. Therefore a zero of the splitting distance defined below is a geometric intersection of $W^s$ and $W^u$, independently of the particular $t$-values corresponding to $\tau=0$ on the two orbits.

The renormalized stable/unstable section is then given by 
\begin{equation}\label{M-section-value}
{\mathbb M}^{s,u}(0)= -\frac{X^{s,u}(0)-\sqrt2}{\varepsilon^3\sqrt\varepsilon\,a(0)} .
\end{equation}

\begin{definition}
We define 
$$
{\mathbb D}(\theta_0, \rho, \varepsilon) := \frac{1}{\rho}\left({\mathbb M}^s(0) - {\mathbb M}^u(0)\right)
$$ 
as the {\bf splitting distance} of the stable manifold $W^s$ and the unstable manifold $W^u$. 
\end{definition}
At $\tau=0$ both parametrizations satisfy $\Theta=0$ and ${\mathbb W}=0$. Moreover, \eqref{M-section-value} gives
\begin{equation}\label{geometric-splitting}
X^s(0)-X^u(0) =-\rho\,\varepsilon^3\sqrt{\varepsilon}\,a(0) {\mathbb D}(\theta_0,\rho,\varepsilon).
\end{equation}
At a zero of ${\mathbb D}$ with $\rho>0$, the section curves,
and equivalently the two-dimensional invariant manifolds inside the
three-dimensional Jacobi surface, meet transversally precisely when
\begin{equation}\label{transverse-criterion}
\partial_{\theta_0}{\mathbb D}(\theta_0,\rho,\varepsilon)\ne0.
\end{equation}
Note also that 
\begin{equation}\label{oddD}
{\mathbb D}(-\theta_0,\rho,\varepsilon)=-{\mathbb D}(\theta_0,\rho,\varepsilon)
\end{equation}
from reversibility. By Proposition~\ref{props2.1}, the numerator defining
$\mathbb D$ is jointly real analytic in $(\theta_0,\rho)$
and vanishes at $\rho=0$. Hence $\mathbb D$ extends jointly
real analytically across $\rho=0$. We expand ${\mathbb D}(\theta_0, \rho, \varepsilon)$ as a power series in $\rho$ at $\rho = 0$ to write
$$
{\mathbb D}(\theta_0, \rho, \varepsilon) = {\mathbb D}_0(\theta_0, \varepsilon) + {\mathbb D}_1(\theta_0, \varepsilon) \rho + {\mathbb D}_2(\theta_0, \varepsilon) \rho^2 + \cdots.
$$

\begin{proposition}[Poincar\'e--Melnikov integral]\label{prop:corrected-melnikov}
Regard $\rho$ as the perturbation parameter and put
\[
\phi=\theta_0+\psi(\tau),\qquad R_1=\sqrt{1-2\varepsilon^2a^2\cos\phi+\varepsilon^4a^4}.
\]
Then
\begin{equation}\label{e.050401}
{\mathbb D}_0(\theta_0,\varepsilon) =\frac{\sqrt2}{4\varepsilon^{3/2}}\,{\mathbb K}(\theta_0,\varepsilon) -\frac{1}{2\varepsilon^{7/2}}\,{\mathbb J}(\theta_0,\varepsilon),
\end{equation}
where
\begin{align}
{\mathbb K} &=\int_{-\infty}^{+\infty} a^5\sin\phi\,(1-R_1^{-3})\,d\tau, \label{e.0219K}\\
{\mathbb J} &=\int_{-\infty}^{+\infty} ab\left[\varepsilon^2a^2\cos\phi\,(2+R_1^{-3}) +(1-R_1^{-3})\right]\,d\tau. \label{e.0219J}
\end{align}
Both integrals are absolutely convergent.
\end{proposition}

\begin{proof}
Fix $\varepsilon$ and the actual section phase $\theta_0$.
At $\rho=0$, both primary branches reduce to
$(X,Y,\theta,\mathcal J)=(a,b,\phi,1)$.
Since $F$ and $G$ vanish identically at $\rho=0$, expansion of
\eqref{e.050302}--\eqref{e.050403} along either branch gives
$F=\rho F_1+O(\rho^2)$ and $G=\rho G_1+O(\rho^2)$ on every compact
$\tau$-interval, where
\begin{align}\label{F-first-rho}
\begin{split}
F_1&=(1-R_1^{-3})+\varepsilon^2a^2\cos\phi\,(2+R_1^{-3}),\\
G_1&=\sin\phi\,(1-R_1^{-3}).
\end{split}
\end{align}
Recall that $H_D(X,Y)=\frac12Y^2-\frac12X^2+\frac14X^4$ is the Duffing Hamiltonian.
Equation \eqref{EqnXYU1} gives the exact identity
\begin{align*}
\frac{dH_D}{d\tau}
&=Y\left(X-X^3-XF+\sqrt2\varepsilon^2\mathcal J^2XYG\right)
 +(X^3-X)\left(Y+\sqrt2\varepsilon^2\mathcal J^2X^2G\right)\\
&=-XYF+\sqrt2\varepsilon^2\mathcal J^2X(Y^2-X^2+X^4)G.
\end{align*}
For real angles and sufficiently small real $\rho$, the barycentric
cancellations in \eqref{e.050302}--\eqref{e.050403} give
$|F|\le C|\rho|\varepsilon^4X^4$ and
$|G|\le C|\rho|\varepsilon^2X^2$ along these branches.
Together with $|X|+|Y|\le Ca$ from
\eqref{global-branch-closeness} and the boundedness of $\mathcal J$,
this yields
\begin{align*}
\frac1{|\rho|}\left|\frac{dH_D}{d\tau}\right|
\le C\varepsilon^4\left(|X|^5|Y|+|X|^3(Y^2+X^2+X^4)\right)
\le C\varepsilon^4a^5.
\end{align*}
This bound is integrable in $\tau$ and independent of $\rho$.

Since $H_D$ tends to zero at the infinite endpoint of each primary
branch, integration gives
\begin{align*}
H_D(X_\Sigma^s,0)-H_D(X_\Sigma^u,0)
=-\int_0^\infty\frac{d}{d\tau}H_D(X^s,Y^s)\,d\tau
 -\int_{-\infty}^0\frac{d}{d\tau}H_D(X^u,Y^u)\,d\tau, 
\end{align*}
where $X_\Sigma^{s,u}=X_0^{s,u}(\theta_0,\rho,\varepsilon)$ is the cross section of the primary stable/unstable solution obtained in Proposition~\ref{props2.1}. 

Divide by $\rho$ and apply dominated convergence. Using
$b^2=a^2-\frac12a^4$ and \eqref{F-first-rho}, we obtain
\begin{align*}
\lim_{\rho\to0}\frac{H_D(X_\Sigma^s,0)-H_D(X_\Sigma^u,0)}{\rho}
=-\int_{-\infty}^{\infty}
\left(-abF_1+\frac{\sqrt2}{2}\varepsilon^2a^5G_1\right)\,d\tau.
\end{align*}
The section coordinates are analytic by
Proposition~\ref{props2.1} and equal $a(0)=\sqrt2$ at $\rho=0$.
Since $\partial_XH_D(\sqrt2,0)=\sqrt2$, the normalization
\eqref{M-section-value} therefore gives
\begin{align*}
\mathbb D_0
&=-\frac1{2\varepsilon^{7/2}}
 \lim_{\rho\to0}\frac{H_D(X_\Sigma^s,0)-H_D(X_\Sigma^u,0)}{\rho}\\
&=\frac1{2\varepsilon^{7/2}}\int_{-\infty}^{\infty}
 \left(-abF_1+\frac{\sqrt2}{2}\varepsilon^2a^5G_1\right)\,d\tau
 =\frac{\sqrt2}{4\varepsilon^{3/2}}\mathbb K
 -\frac1{2\varepsilon^{7/2}}\mathbb J.
\end{align*}
This proves \eqref{e.050401}. Finally,
$F_1=O(\varepsilon^4a^4)$ and $G_1=O(\varepsilon^2a^2)$,
so both integrals defining $\mathbb K$ and $\mathbb J$ are
absolutely convergent.
\end{proof}

\subsection{Fourier expansion of the first approximation}\label{s3.2}
We write \eqref{e.050401} as
\begin{equation}\label{e.042301}
{\mathbb D}_0(\theta_0) =\frac{\sqrt2}{4\varepsilon^{3/2}}\,{\mathbb K} -\frac{1}{2\varepsilon^{7/2}}\,{\mathbb J},
\end{equation}
where ${\mathbb K}$ and ${\mathbb J}$ are given by
\eqref{e.0219K}--\eqref{e.0219J}. For later use we introduce for $j=1,2$, 
\begin{align}\label{e.051003}
\begin{split}
I^{(j)}_{n,m,\odd} &=\varepsilon^{4n}\int_{-\infty}^{+\infty} a^{4n+4+(-1)^j}e^{\mi(2m+1)\psi}\,d\tau, \qquad {\mathbb I}^{(j)}_{n,m,\odd}=\operatorname{Re}I^{(j)}_{n,m,\odd},\\
I^{(j)}_{n,m,\even} &=\varepsilon^{4n-2}\int_{-\infty}^{+\infty} a^{4n+2+(-1)^j}e^{2\mi m\psi}\,d\tau, \qquad {\mathbb I}^{(j)}_{n,m,\even}=\operatorname{Re}I^{(j)}_{n,m,\even},\\
J_{n,m,\odd} &=\varepsilon^{4n+2}\int_{-\infty}^{+\infty} a^{4n+3}b\,e^{\mi(2m+1)\psi}\,d\tau, \qquad {\mathbb J}_{n,m,\odd}=\operatorname{Im}J_{n,m,\odd},\\
J_{n,m,\even} &=\varepsilon^{4n}\int_{-\infty}^{+\infty} a^{4n+1}b\,e^{2\mi m\psi}\,d\tau, \qquad {\mathbb J}_{n,m,\even}=\operatorname{Im}J_{n,m,\even}.
\end{split}
\end{align}

The Fourier coefficients below are collected in
Lemmas \ref{l.051001} and \ref{l.051002} of
Appendix \ref{a.051001}.

\begin{proposition}\label{p.051001}
The two functions in \eqref{e.042301} have sine expansions
\[
{\mathbb K} =\sum_{m=0}^{\infty}{\mathbb K}^{\odd}_m \sin(2m+1)\theta_0 +\sum_{m=1}^{\infty}{\mathbb K}^{\even}_m \sin(2m)\theta_0,
\]
and
\[
{\mathbb J} =\sum_{m=0}^{\infty}{\mathbb J}^{\odd}_m \sin(2m+1)\theta_0 +\sum_{m=1}^{\infty}{\mathbb J}^{\even}_m \sin(2m)\theta_0.
\]
More precisely,
\begin{align*}
{\mathbb K}^{\odd}_0 &=-\sum_{k=1}^{\infty}\sum_{\ell=0}^{\infty} C_{0,k,\ell}\, {\mathbb I}^{(2)}_{k+\ell,0,\odd},\\
{\mathbb K}^{\odd}_m &=-\sum_{k=m}^{\infty}\sum_{\ell=0}^{\infty} C_{m,k,\ell}\, {\mathbb I}^{(2)}_{k+\ell,m,\odd}, \qquad m\geq1,\\
{\mathbb K}^{\even}_m &=\sum_{k=m}^{\infty}\sum_{\ell=0}^{\infty} E_{m,k,\ell}\, {\mathbb I}^{(2)}_{k+\ell,m,\even}, \qquad m\geq1,
\end{align*}
where $C_{m,k,\ell},E_{m,k,\ell}$ are given in
Lemma \ref{l.051001}. Furthermore,
\begin{align*}
{\mathbb J}^{\odd}_0 &=\sum_{k=1}^{\infty}\sum_{\ell=0}^{\infty} {\mathcal C}_{0,k,\ell}\, {\mathbb J}_{k+\ell,0,\odd},\\
{\mathbb J}^{\odd}_m &=\sum_{k=m}^{\infty}\sum_{\ell=0}^{\infty} {\mathcal C}_{m,k,\ell}\, {\mathbb J}_{k+\ell,m,\odd}, \qquad m\geq1,\\
{\mathbb J}^{\even}_m &=\sum_{k=m}^{\infty}\sum_{\ell=0}^{\infty} {\mathcal E}_{m,k,\ell}\, {\mathbb J}_{k+\ell,m,\even}, \qquad m\geq1,
\end{align*}
where ${\mathcal C}_{m,k,\ell},{\mathcal E}_{m,k,\ell}$ are given in
Lemma \ref{l.051002}.
\end{proposition}
The proposition shows that the task of evaluating $\mathbb{D}_0\left(\theta_0\right)$ is essentially reduced to evaluating $I_{n, m, \text { odd }}^{(j)}$ and $J_{n, m, \text { odd }}$ for $(n, m) \in \Delta_{\text {odd }}$, and $I_{n, m, \text { even }}^{(j)}, J_{n, m, \text { even }}$ for $(n, m) \in \Delta_{\text {even }}$.
Here 
\begin{align*}
\begin{split}
\Delta_{\odd} &=\{(n,0):n\geq1\} \cup\{(n,m):m\geq1,\ n\geq m\},\\
\Delta_{\even} &=\{(n,m):m\geq1,\ n\geq m\}.
\end{split}
\end{align*}
The first-harmonic coefficients which will dominate the splitting are
\begin{align*}
-C_{0,1,0}=-\frac38,\qquad {\mathcal C}_{0,1,0}=\frac32.
\end{align*}

\subsection{Estimates on the Fourier coefficients}  \label{s3.3}
 Recall from \eqref{e.050303} that 
\begin{equation*}
\begin{split}
a(\tau) = & \frac{2 \sqrt{2}}{ e^{\tau}+e^{-\tau}}, \ \ \ \ \ \ \ \ b(\tau) = \frac{2 \sqrt{2} \left(e^{-\tau}-e^{\tau}\right)}{\left(e^{\tau}+e^{-\tau}\right)^2},\\
\psi(\tau) = & 2 \tan^{-1} \frac{1}{2}(e^{\tau}- e^{-\tau}) -\frac{1}{48\varepsilon^3} \left(e^{3 \tau}- e^{-3\tau}\right)-\frac{3}{16\varepsilon^3}\left(e^{\tau}- e^{-\tau}\right).
\end{split}
\end{equation*}
 
Let $y: {\mathbb R} \to {\mathbb R}$ be such that 
\begin{equation}\label{ztoy}
y(z) = (\sqrt{z^2 + 4^3}+z)^{1/3} -( \sqrt{z^2 + 4^3} -z)^{1/3}.
\end{equation}

In view of \eqref{e.051003}, using the formula 
\begin{align*}
\int_{-\infty}^{\infty}a^pb^re^{\mi q\psi}d\tau = \frac{2}{3}\int_{-\infty}^{\infty}\frac{ (2 \sqrt{2})^{p+r}(-1)^{r+q} e^{- \frac{\mi q}{24\varepsilon^3}z}y(z)^r}{(y(z)+2\mi)^{1+q+(p+2r+1)/2}(y(z)-2\mi)^{1-q+(p+2r+1)/2}}dz
\end{align*}
proved in Lemma \ref{l.043001} in Appendix \ref{a.051001}, we have 
\begin{proposition}\label{p.042701}
Let $\delta(j) = \frac{1+(-1)^j}{2}$ for $j=1,2$. Then 
\begin{align*}
\begin{split}
I^{(j)}_{n,m,\odd} &= \frac{2(2 \sqrt{2})^{4(n+1)+(-1)^j}\ee^{4n}}{-3}\int_{-\infty}^{\infty}\frac{e^{- \frac{\mi(2m+1)}{24\varepsilon^3}z}}{(y(z)+2\mi)^{2(n+m+2)+\delta(j)}(y(z)-2\mi)^{2(n-m+1)+\delta(j)}}dz,\\
I^{(j)}_{n,m,\even}& = \frac{2(2 \sqrt{2})^{4n+2+(-1)^j}\varepsilon^{4n-2}}{3}\int_{-\infty}^{\infty}\frac{e^{- \frac{\mi m}{12\varepsilon^3}z}}{(y(z)+2\mi)^{2(n+m+1)+\delta(j)}(y(z)-2\mi)^{2(n-m+1)+\delta(j)}}dz,\\
J_{n,m,\odd}& =\frac{2(2 \sqrt{2})^{4(n+1)}\varepsilon^{4n+2} }{3}\int_{-\infty}^{\infty}\frac{e^{- \frac{\mi(2m+1)}{24\varepsilon^3}z}y(z)}{(y(z)+2\mi)^{2(n+m)+5}(y(z)-2\mi)^{2(n-m)+3}}dz,\\
J_{n,m,\even}& = \frac{2(2 \sqrt{2})^{4n+2}\varepsilon^{4n}}{-3}\int_{-\infty}^{\infty}\frac{e^{- \frac{\mi m}{12\varepsilon^3}z}y(z)}{(y(z)+2\mi)^{2(n+m)+3}(y(z)-2\mi)^{2(n-m)+3}}dz.
\end{split}
\end{align*}
\end{proposition}

\vskip0.05in

To estimate the integrals in Proposition \ref{p.042701}, we introduce 
\begin{align*}
\mathcal{I}_{m,n,r}(k) = \int_{-\infty}^{\infty}\frac{e^{-\mi kz}y(z)^r}{(y(z)+2\mi)^m(y(z)-2\mi)^n} dz,
\end{align*}
where $k>0$ and $m,n,r\in\mathbb N$ such that the integral converges, and $y(z)$ is as defined in \eqref{ztoy}. 
\begin{lemma}\label{l.042701}
	As $k\to\infty$, we have 
	\begin{align*}
\mathcal{I}_{m,n,r}(k) =\mi^{3(r-n)-m}2^{r-2n+1} \frac{3^{\frac{m}{2}}\pi}{\Gamma(m/2)}e^{-8k} k^{\frac{m-2}{2}}+O(k^{\frac{m-3}{2}}e^{-8k}).
\end{align*}	
\end{lemma}
\begin{proof}
	We choose $(-\infty \mi, -8\mi)$  and $(8\mi, \infty \mi)$ as branch cuts, to extend the function $y(z)$ to the complex domain.  With $\zeta=z+8\mi$, we have
	$y=y_0+y_{1/2}\zeta^{1/2}+O(\zeta)$, where
	$y_0=-2\mi$ and $y_{1/2}=(1+\mi)/\sqrt6$. Hence
	\begin{align*}
	\frac{y(z)^r}{(y(z)+2\mi)^m(y(z)-2\mi)^n}
	=Y_0\zeta^{-m/2}\left(1+O(\zeta^{1/2})\right),
	\qquad Y_0=y_0^r y_{1/2}^{-m}(-4\mi)^{-n}.
	\end{align*}
	By Watson's lemma \cite{NR86}, the dominating term of the integral is from the branch point $z=-8\mi$, which is reduced to 
	\begin{align*}
\mathcal{I}_0(k) := -Y_0e^{-8k} \int_{H(-0\mi)}\zeta^{-m/2} e^{-\mi k\zeta} d\zeta,
\end{align*}
	where $H(-0\mi)$ is the Hankel contour enclosing the negative imaginary axis. By a change of variable $u=-\mi k\zeta$, this contour is transformed to $H(-0)$, the Hankel contour enclosing the negative real axis, which gives 
	\begin{align*}
\mathcal{I}_0(k) = -Y_0e^{-8k}(-\mi k)^{\frac{m-2}{2}} \int_{H(-0)}u^{-m/2} e^{u} du = -Y_0e^{-8k}(-\mi k)^{\frac{m-2}{2}} \frac{2\pi \mi}{\Gamma(m/2)}.
\end{align*}
	Here we used the integral representation of the reciprocal Gamma function 
	\[\frac{1}{2 \pi \mi} \int_{H(-0)} s^{-z} e^s d s=\frac{1}{\Gamma(z)}, z\in\mathbb C.\]
	Therefore,  
	\begin{align*}
\mathcal{I}_{m,n,r}(k) &= -(-2\mi)^r(-4\mi)^{-n}\left(\frac{\sqrt{6}}{6}(1+\mi)\right)^{-m} \frac{2\pi \mi (-\mi)^{\frac{m-2}{2}}}{\Gamma(m/2)} e^{-8k} k^{\frac{m-2}{2}} +O(k^{\frac{m-3}{2}}e^{-8k})\\
& =(-1)^{r-n-\frac{m}{2}}2^{r+1-2n}(\sqrt{3})^m e^{-\mi \frac{\pi m}{4}}\mi^{r-n+\frac{m}{2}}\frac{\pi}{\Gamma(m/2)}e^{-8k} k^{\frac{m-2}{2}}+O(k^{\frac{m-3}{2}}e^{-8k})\\
& =\mi^{3(r-n)-m}2^{r+1-2n}(\sqrt{3})^m \frac{\pi}{\Gamma(m/2)}e^{-8k} k^{\frac{m-2}{2}}+O(k^{\frac{m-3}{2}}e^{-8k}).
\end{align*}
\end{proof}
\begin{remark}
	By Lemma \ref{l.042701} and similar transformations used in Lemma \ref{l.043001}, it follows that 
	\begin{align*}
\int_{-\infty}^{\infty}\frac{e^{- \frac{\mi}{\varepsilon^3}\left(\frac{z^3}{3}+z\right)}}{(z+\mi)^{4}}dz &= \frac{2^4}{3}\int_{-\infty}^{\infty}\frac{e^{-\frac{\mi z}{12\varepsilon^3}}}{(y(z)+2\mi)^{5}(y(z)-2\mi)}dz\\
&= \frac{4\sqrt{\pi}}{3}\varepsilon^{-\frac{9}{2}}e^{-\frac{2}{3\varepsilon^3}} + O(e^{-\frac{2}{3\varepsilon^3}}\varepsilon^{-3}),
\end{align*}
	and 
	\begin{align*}
\int_{-\infty}^{\infty}\frac{e^{- \frac{\mi}{2\varepsilon^3}\left(\frac{z^3}{3}+z\right)}}{(z+\mi)^{4}(z-\mi)^2}dz &= \frac{2^6}{3}\int_{-\infty}^{\infty}\frac{e^{-\frac{\mi z}{24\varepsilon^3}}}{(y(z)+2\mi)^{5}(y(z)-2\mi)^{3}}dz\\
& = -\frac{\sqrt{2 \pi}}{12} \varepsilon^{-9/2}e^{-\frac{1}{3\varepsilon^3}} + O(\varepsilon^{-3}e^{-\frac{1}{3\varepsilon^3}}),
\end{align*}
	which recovers the relevant estimates in \cite{G}. 
\end{remark}

In view of \eqref{e.042301}, Proposition \ref{p.042701} and Lemma
\ref{l.042701} give the following fixed-mode estimates.
\begin{corollary}\label{c.042701}
For every fixed admissible pair $(n,m)$,
\begin{align*}
\varepsilon^{-3/2}{\mathbb I}^{(j)}_{n,m,\odd} &\asymp e^{-(2m+1)/(3\varepsilon^3)} \varepsilon^{n-3m-9/2-(3/2)\delta(j)},\\
\varepsilon^{-3/2}{\mathbb I}^{(j)}_{n,m,\even} &\asymp e^{-2m/(3\varepsilon^3)} \varepsilon^{n-3m-7/2-(3/2)\delta(j)},\\
\varepsilon^{-7/2}{\mathbb J}_{n,m,\odd} &\asymp e^{-(2m+1)/(3\varepsilon^3)} \varepsilon^{n-3m-6},\qquad \varepsilon^{-7/2}{\mathbb J}_{n,m,\even} \asymp e^{-2m/(3\varepsilon^3)} \varepsilon^{n-3m-5}.
\end{align*}
\end{corollary}

The following estimate controls the Fourier and binomial tails
uniformly in their indices.

\begin{lemma}\label{lem:uniform-harmonic-tails}
Fix integers $p\ge1$ and $r\ge0$, and define
\begin{align*}
T^{p,r}_{N,m}(\varepsilon)
=\varepsilon^{2N}\int_{\mathbb R}a^{2N+p}b^r e^{\mi m\psi}\,d\tau,
\qquad N\ge0,\quad m\ge1.
\end{align*}
There exist $C>1$ and $\varepsilon_2>0$, depending only on $p,r$,
such that
\begin{align}\label{uniform-harmonic-bound}
|T^{p,r}_{N,m}(\varepsilon)|
\le C^{N+m}
\varepsilon^{\frac N2-\frac32m-\frac34(p+2r-1)}
e^{-m/(3\varepsilon^3)},
\qquad 0<\varepsilon<\varepsilon_2.
\end{align}
\end{lemma}

\begin{proof}
Set $\zeta=\sinh\tau$ and $d=N+(p+2r+1)/2$.
The formulas in \eqref{e.050303} give
\begin{align*}
T^{p,r}_{N,m}
=(-1)^{r+m}2^{N+(p+r)/2}\varepsilon^{2N}
\int_{\mathbb R}
\frac{\zeta^r}{(1+\zeta^2)^d}
\left(\frac{\zeta-\mi}{\zeta+\mi}\right)^m
e^{-\frac{\mi m}{2\varepsilon^3}(\zeta+\zeta^3/3)}\,d\zeta.
\end{align*}
Choose the branch of $(1+\zeta^2)^{-d}$ that is positive on
the real axis. The integrand is holomorphic in
$-1<\operatorname{Im}\zeta<1$.
Putting $h=\varepsilon^{3/2}<1/2$ and shifting the contour to
$\zeta=x-\mi(1-h)$, the vertical sides vanish at infinity.
On this line,
$|\zeta+\mi|\ge h$,
$|\zeta-\mi|\ge\max\{1,|\zeta|\}$, and
$|(\zeta-\mi)/(\zeta+\mi)|\le2/h$.
Since $d\ge r+1$, we obtain
\begin{align*}
\left|
\frac{\zeta^r}{(1+\zeta^2)^d}
\left(\frac{\zeta-\mi}{\zeta+\mi}\right)^m
\right|
&\le 2^m h^{-d-m},\\
\left|e^{-\frac{\mi m}{2\varepsilon^3}(\zeta+\zeta^3/3)}\right|
&\le e^{m/2}e^{-m/(3\varepsilon^3)}
e^{-mx^2/(4\varepsilon^3)}.
\end{align*}
Thus the Gaussian integral gives
\begin{align*}
|T^{p,r}_{N,m}|
&\le C^{N+m}\varepsilon^{2N}h^{-d-m}
e^{-m/(3\varepsilon^3)}
\int_{\mathbb R}e^{-mx^2/(4\varepsilon^3)}\,dx\\
&\le C^{N+m}
\varepsilon^{\frac N2-\frac32m-\frac34(p+2r-1)}
e^{-m/(3\varepsilon^3)},
\end{align*}
which proves \eqref{uniform-harmonic-bound}.
\end{proof}

\begin{proposition}\label{p.043001}\label{prop4.8}
The splitting distance in \eqref{e.042301} satisfies, in the
$C^1({\mathbb R}/2\pi{\mathbb Z})$ norm,
\begin{equation}\label{corrected-leading-splitting}
{\mathbb D}_0(\theta_0,\varepsilon) =-\frac18\sqrt{\frac\pi2}\, \varepsilon^{-5}e^{-1/(3\varepsilon^3)}\sin\theta_0 +O_{C^1}\!\left( \varepsilon^{-9/2}e^{-1/(3\varepsilon^3)}\right).
\end{equation}
\end{proposition}

\begin{proof}
By Proposition~\ref{p.051001}, the terms with
$(m,k,\ell)=(0,1,0)$ have coefficients
$-C_{0,1,0}=-3/8$ and ${\mathcal C}_{0,1,0}=3/2$.
The coefficient formulas in Lemmas~\ref{l.051001}
and~\ref{l.051002} grow at most exponentially in $k+\ell$. The remaining first-harmonic terms have $k+\ell\ge2$,
corresponding to $N\ge4$ for ${\mathbb K}$ and $N\ge5$
for ${\mathbb J}$ in Lemma~\ref{lem:uniform-harmonic-tails}.
All other terms have frequency at least two.
Applying that lemma with $(p,r)=(5,0)$ and $(1,1)$,
respectively, and summing the geometric bounds, also after
one phase derivative, gives
\begin{align}
{\mathbb K}
&=-\frac38{\mathbb I}^{(2)}_{1,0,\odd}\sin\theta_0
+O_{C^1}\!\left(\varepsilon^{-3}e^{-1/(3\varepsilon^3)}\right),
\label{K-leading-uniform}\\
{\mathbb J}
&=\frac32{\mathbb J}_{1,0,\odd}\sin\theta_0
+O_{C^1}\!\left(\varepsilon^{-1}e^{-1/(3\varepsilon^3)}\right).
\label{J-leading-uniform}
\end{align}
Proposition~\ref{p.042701} and Lemma~\ref{l.042701} yield
\begin{align*}
{\mathbb I}^{(2)}_{1,0,\odd}
&=\frac{2\sqrt\pi}{15}\varepsilon^{-7/2}e^{-1/(3\varepsilon^3)}
+O\!\left(\varepsilon^{-2}e^{-1/(3\varepsilon^3)}\right),\\
{\mathbb J}_{1,0,\odd}
&=\frac{2}{15}\sqrt{\frac\pi2}\varepsilon^{-3/2}e^{-1/(3\varepsilon^3)}
+O\!\left(e^{-1/(3\varepsilon^3)}\right).
\end{align*}
Substituting these estimates and
\eqref{K-leading-uniform}--\eqref{J-leading-uniform}
into \eqref{e.042301}, we obtain
\begin{align*}
{\mathbb D}_0(\theta_0,\varepsilon)
&=-\left(\frac1{40}+\frac1{10}\right)\sqrt{\frac\pi2}
\varepsilon^{-5}e^{-1/(3\varepsilon^3)}\sin\theta_0
+O_{C^1}\!\left(\varepsilon^{-9/2}e^{-1/(3\varepsilon^3)}\right)\\
&=-\frac18\sqrt{\frac\pi2}\varepsilon^{-5}e^{-1/(3\varepsilon^3)}
\sin\theta_0
+O_{C^1}\!\left(\varepsilon^{-9/2}e^{-1/(3\varepsilon^3)}\right),
\end{align*}
which proves \eqref{corrected-leading-splitting}.
\end{proof}

\medskip

We are now ready to prove the main result of this section. 
\begin{proof}[Proof of Proposition \ref{prop3.1}]
Note that 
${\mathbb D}(0,\rho,\varepsilon)=0$ by \eqref{oddD} for every  $\rho\in(0,\frac12]$.  By
Proposition \ref{p.043001},
\[
\partial_{\theta_0}{\mathbb D}(0,0,\varepsilon) =-\frac18\sqrt{\frac\pi2}\varepsilon^{-5} e^{-1/(3\varepsilon^3)} \left(1+O(\sqrt\varepsilon)\right),
\]
which is nonzero for $0<\varepsilon<\varepsilon_1$ after decreasing
$\varepsilon_1$.  Proposition \ref{props2.1} implies that
$\partial_{\theta_0}{\mathbb D}(0,\rho,\varepsilon)$ is real analytic, hence
continuous in $\rho$ at $0$.  It remains nonzero for
$0<\rho<\rho(\varepsilon)$ with $\rho(\varepsilon)>0$ sufficiently small.
The geometric relation \eqref{geometric-splitting} and the transversality
criterion \eqref{transverse-criterion} then give a transversal homoclinic
intersection for every such positive mass ratio.
\end{proof}

\begin{proof}[Proof of the Main Theorem]
Choose $\varepsilon_1>0$ small enough so that Proposition
\ref{props2.1} and Proposition
\ref{p.043001} hold for all $0<\varepsilon<\varepsilon_1$, and set
$J_0=-\varepsilon_1^{-1}$.  Fix $J<J_0$ and put
$\varepsilon=|J|^{-1}$.  For this fixed $\varepsilon$  Proposition \ref{props2.1} shows that $\partial_{\theta_0}{\mathbb D}(0,\rho,\varepsilon)$ is real analytic in $\rho$
on a neighborhood $\mathcal O$ of the compact interval $[0,1/2]$, while Proposition
\ref{p.043001} gives $\partial_{\theta_0}{\mathbb D}_0(0,\varepsilon)\ne0$. 
Thus $\partial_{\theta_0}{\mathbb D}(0,\rho,\varepsilon)$   is not identically zero.  Since it is analytic on $\mathcal O$, its zero set $E_\varepsilon$ in the physical interval $(0,1/2]\subset \mathcal O$ is finite.   For every $\rho\in(0,1/2]\setminus E_\varepsilon$, \eqref{oddD} gives ${\mathbb D}(0,\rho,\varepsilon)=0$, while the
definition of $E_\varepsilon$ gives
$\partial_{\theta_0}{\mathbb D}(0,\rho,\varepsilon)\ne0$.  By
\eqref{geometric-splitting} and \eqref{transverse-criterion}, $W^s$ and
$W^u$ have a transversal homoclinic intersection. 
\end{proof}

\appendix

\section{Expansions in \texorpdfstring{$\ee$}{epsilon}}\label{s.050501}
In this section, we collect several relevant series expansions. Our main tool is the binomial theorem, 
\begin{equation*}
(1+x)^{\alpha} = 1 +\alpha x + \sum_{n=2}^{\infty} \binom{\alpha}{n} x^n, \quad |x| < 1,
\end{equation*}
where the generalized binomial coefficient is defined by
\[
\binom{\alpha}{n}=\frac{\Gamma(\alpha+1)}{\Gamma(n+1) \Gamma(\alpha-n+1)}= \frac{1}{n!} \alpha (\alpha-1)\cdots (\alpha - (n-1)).
\]
In what follows, we assume that $\varepsilon_0$ is sufficiently small so that the condition $|x| < 1$ is always satisfied whenever this expansion is applied. 

\vskip0.05in

Moreover, we use $O(1)$ to denote a generic real analytic function in the variables $X, Y, \theta$ and $\rho, \varepsilon$, where $(\rho, \varepsilon) \in D_{\ee_0} = (-\varepsilon_0, 1/2 + \varepsilon_0) \times (0, \varepsilon_0)$, and $(X, Y) \in D_{\ell}^+$ (see the definition of $D_{\ell}^+$ below \eqref{e.050402}), and $\theta_0 \in \mathbb{R}$, such that the $C^1$-norm of this function is uniformly bounded by a constant independent of the parameters.

\subsection{Expansions of mutual distances} For the reader's convenience, we recall the formulas of $R_{13}$ and $R_{23}$ from \eqref{e.050403}, 
\begin{equation*}
\begin{split}
R_{13} = & \sqrt{1+\rho^2\varepsilon^4 \mathcal{J}^4 X^4 + 2 \rho \varepsilon^2 \mathcal{J}^2 X^2 \cos \theta },\\
R_{23} = & \sqrt{1 + (1-\rho)^2 \varepsilon^4 \mathcal{J}^4 X^4 -2 (1-\rho) \varepsilon^2 \mathcal{J}^2 X^2 \cos \theta }.
\end{split}
\end{equation*}
We have 
\begin{align}\label{e.050404}
\begin{split}
R_{13}^{-k} & = 1-\frac{k}{2}\left(\rho^2\varepsilon^4 \mathcal{J}^4 X^4 + 2 \rho \varepsilon^2 \mathcal{J}^2 X^2 \cos \theta\right) + \sum_{j=2}^{\infty}\binom{-\frac{k}{2}}{j}\left(\rho^2\varepsilon^4 \mathcal{J}^4 X^4 + 2 \rho \varepsilon^2 \mathcal{J}^2 X^2 \cos \theta\right)^{j}\\
& = 1- k\rho \varepsilon^2 \mathcal{J}^2 X^2 \cos \theta +\left(-\frac{k}{2}+4\binom{-\frac{k}{2}}{2}\cos^2\theta\right)\rho^2\varepsilon^4 \mathcal{J}^4 X^4 +\ee^6\mathcal{J}^6X^6O(1),
\end{split}
\end{align}
and similarly, 
\begin{align}\label{e.050405}
R_{23}^{-k}\!=\! 1+k(1-\rho) \varepsilon^2 \mathcal{J}^2 X^2 \cos \theta +\left(-\frac{k}{2}+4\binom{-\frac{k}{2}}{2}\cos^2\theta\right)(1-\rho)^2\varepsilon^4 \mathcal{J}^4 X^4+ \ee^6\mathcal{J}^6X^6O(1).
\end{align}
\subsection{Expansions of gravitational functions}
By equation \eqref{Jacobi1} for the Jacobi integral and expansions \eqref{e.050404}-\eqref{e.050405} with $k=1$, we have 
\begin{align*}
\mathcal{J} & = 1- \varepsilon^3 \mathcal{J}^3 Y^2 - \frac{1}{2} \varepsilon^3 \mathcal{J}^3 X^4 + (1-\rho) \varepsilon^3 \mathcal{J}^3 X^2 R_{13}^{-1}+ \rho \varepsilon^3 \mathcal{J}^3 X^2 R_{23}^{-1}\\
&= 1+ \ee^3\mathcal{J}^3\left(X^2-Y^2-\frac12X^4\right) + \left(-\frac{1}{2}\sin^2\theta+\cos^2\theta\right)\rho(1-\rho)\ee^7\mathcal{J}^7X^6+ \ee^9\mathcal{J}^9X^8O(1).
\end{align*}
Furthermore, by letting 
\begin{align*}
\mathcal{J}_3 &= X^2-Y^2-\frac12X^4,\qquad \mathcal{J}_7 = \left(-\frac{1}{2}\sin^2\theta+\cos^2\theta\right)\rho(1-\rho)X^6,
\end{align*}
one obtains the following expansion with coefficients independent of $\mathcal{J}$ itself, 
\begin{align}\label{e.060406}
\mathcal{J}^k = 1+k\mathcal{J}_3\ee^3+\frac{k(k+5)}{2}\mathcal{J}_3^2\ee^6+k\mathcal{J}_7\ee^7+\left(\mathcal{J}_3^2O(1) + X^8O(1)\right)\ee^9.
\end{align}

Next, it follows from \eqref{e.050302}, expansions \eqref{e.050404}-\eqref{e.050405} with $k=3$ and \eqref{e.060406} that 
\begin{align*}
F &= 1 - (1-\rho) \left(1 +\rho \varepsilon^2 \mathcal{J}^2 X^2 \cos \theta\right)R_{13}^{-3} - \rho \left(1 -(1-\rho) \varepsilon^2 \mathcal{J}^2 X^2 \cos \theta \right)R_{23}^{-3}\\
&= \frac32\rho(1-\rho)\left(1-3\cos^2\theta\right)\ee^4X^4+\ee^7X^4\left(X^2-Y^2-\frac{1}{2}X^4\right)O(1) +\ee^6X^6O(1),
\end{align*}
and 
\begin{align*}
G &= \rho(1-\rho) \sin \theta \left(R_{13}^{-3} - R_{23}^{-3} \right)\\
& = - 3\rho(1-\rho) \sin \theta \cos \theta \varepsilon^2 X^2 + \rho(1-\rho) \sin \theta \left( -\frac{3}{2}\sin^2\theta + 6\cos^2\theta\right)(2\rho-1)\varepsilon^4 X^4\\
&\quad -3\rho(1-\rho) \sin \theta \cos \theta\left(2(X^2-Y^2)-X^4\right)\ee^5X^2 + \ee^6X^2O(1).
\end{align*}

\section{Lemmas for evaluating the Poincar\'e--Melnikov integral}\label{a.051001}
In this section we collect the Fourier identities used in Section
\ref{s3.2}.  Recall that $(a,b,\psi)$ is the homoclinic orbit in
\eqref{e.050303}, put
\[
\phi=\theta_0+\psi,\qquad q=\varepsilon^2a^2, \qquad R_1=(1-2q\cos\phi+q^2)^{1/2},
\]
and let ${\mathbb K},{\mathbb J}$ be as in
\eqref{e.0219K}--\eqref{e.0219J}.  We use the convention
$\binom{\alpha}{n}=0$ for $n<0$.

\begin{lemma}\label{l.050901}
One has
\begin{align}\label{correct-R-Fourier}
\begin{split}
R_1^{-3} &=\frac1{(1+q^2)^{3/2}} +\sum_{k=1}^{\infty} \binom{-3/2}{2k}\frac{q^{2k}}{(1+q^2)^{2k+3/2}} \binom{2k}{k}\\
&\quad+2\sum_{m=1}^{\infty}\sum_{k=m}^{\infty} \binom{-3/2}{2k}\frac{q^{2k}}{(1+q^2)^{2k+3/2}} \binom{2k}{k-m}\cos2m\phi\\
&\quad-2\sum_{m=0}^{\infty}\sum_{k=m}^{\infty} \binom{-3/2}{2k+1}\frac{q^{2k+1}}{(1+q^2)^{2k+5/2}} \binom{2k+1}{k-m}\cos(2m+1)\phi .
\end{split}
\end{align}
In particular, the zero Fourier mode of an even power carries no factor
$2$.
\end{lemma}

\begin{proof}
Write
\[
R_1^{-3}=(1+q^2)^{-3/2} \sum_{n=0}^{\infty}\binom{-3/2}{n} \left(-\frac{q(e^{\mi\phi}+e^{-\mi\phi})}{1+q^2}\right)^n.
\]
For an even power,
\[
(e^{\mi\phi}+e^{-\mi\phi})^{2k} =\binom{2k}{k} +2\sum_{m=1}^{k}\binom{2k}{k-m}\cos2m\phi,
\]
whereas for an odd power,
\[
(e^{\mi\phi}+e^{-\mi\phi})^{2k+1} =2\sum_{m=0}^{k}\binom{2k+1}{k-m} \cos(2m+1)\phi.
\]
Substitution and a change in the order of summation give
\eqref{correct-R-Fourier}.
\end{proof}

\begin{lemma}\label{l.051001}
The Fourier expansion of ${\mathbb K}$ is
\[
{\mathbb K}=\sum_{m=0}^{\infty}{\mathbb K}^{\odd}_m \sin(2m+1)\theta_0 +\sum_{m=1}^{\infty}{\mathbb K}^{\even}_m \sin2m\theta_0,
\]
where
\begin{align*}
{\mathbb K}^{\odd}_0 &=-\sum_{k=1}^{\infty}\sum_{\ell=0}^{\infty} C_{0,k,\ell}{\mathbb I}^{(2)}_{k+\ell,0,\odd},\\
{\mathbb K}^{\odd}_m &=-\sum_{k=m}^{\infty}\sum_{\ell=0}^{\infty} C_{m,k,\ell}{\mathbb I}^{(2)}_{k+\ell,m,\odd}, \qquad m\geq1,\\
{\mathbb K}^{\even}_m &=\sum_{k=m}^{\infty}\sum_{\ell=0}^{\infty} E_{m,k,\ell}{\mathbb I}^{(2)}_{k+\ell,m,\even}, \qquad m\geq1.
\end{align*}
The coefficients are
\begin{align*}
C_{0,k,0} &=\binom{-3/2}{k} +\binom{-3/2}{2k} \left[\binom{2k}{k}-\binom{2k}{k-1}\right],\\
C_{0,k,\ell} &=\binom{-3/2}{2k} \left[\binom{2k}{k}-\binom{2k}{k-1}\right] \binom{-2k-3/2}{\ell},\qquad \ell\geq1,\\
C_{m,k,\ell} &=\binom{-3/2}{2k} \left[\binom{2k}{k-m}-\binom{2k}{k-m-1}\right] \binom{-2k-3/2}{\ell},\\
E_{m,k,\ell} &=\binom{-3/2}{2k-1} \left[\binom{2k-1}{k-m}-\binom{2k-1}{k-m-1}\right] \binom{-2k-1/2}{\ell}.
\end{align*}
In particular, $C_{0,1,0}=3/8$.
\end{lemma}

\begin{proof}
Set $A=1+q^2$ and write
\[
R_1^{-3}=r_0(q)+\sum_{n=1}^{\infty}r_n(q)\cos n\phi.
\]
Lemma \ref{l.050901} gives explicitly
\begin{align}
r_0 &=A^{-3/2}+\sum_{k=1}^{\infty} \binom{-3/2}{2k}\binom{2k}{k} q^{2k}A^{-2k-3/2}, \label{r0-explicit}\\
r_{2m} &=2\sum_{k=m}^{\infty} \binom{-3/2}{2k}\binom{2k}{k-m} q^{2k}A^{-2k-3/2},\qquad m\geq1, \label{r-even-explicit}\\
r_{2m+1} &=-2\sum_{k=m}^{\infty} \binom{-3/2}{2k+1}\binom{2k+1}{k-m} q^{2k+1}A^{-2k-5/2},\qquad m\geq0. \label{r-odd-explicit}
\end{align}
The product identity
\[
\sin\phi\cos n\phi =\frac12\bigl[\sin(n+1)\phi-\sin(n-1)\phi\bigr]
\]
shows that the coefficient of the first harmonic is
\begin{align*}
1-r_0+\frac12r_2 &=1-A^{-3/2} -\sum_{k=1}^{\infty} \binom{-3/2}{2k} \left[\binom{2k}{k}-\binom{2k}{k-1}\right] q^{2k}A^{-2k-3/2}.
\end{align*}
Expanding
\[
1-A^{-3/2}=-\sum_{k=1}^{\infty}\binom{-3/2}{k}q^{2k}, \qquad A^{-s}=\sum_{\ell=0}^{\infty}\binom{-s}{\ell}q^{2\ell},
\]
we obtain
\begin{align}
1-r_0+\frac12r_2 &=-\sum_{k=1}^{\infty}C_{0,k,0}q^{2k} -\sum_{k=1}^{\infty}\sum_{\ell=1}^{\infty} C_{0,k,\ell}q^{2(k+\ell)}. \label{K-first-harmonic-expanded}
\end{align}
This is the first formula in the statement after multiplication by $a^5$
and integration.

For an odd harmonic of order $2m+1$, $m\geq1$, the same product identity
gives
\begin{align}
\begin{split}
-\frac12(r_{2m}-r_{2m+2}) &=-\sum_{k=m}^{\infty}\binom{-3/2}{2k} \left[\binom{2k}{k-m}-\binom{2k}{k-m-1}\right] \times q^{2k}A^{-2k-3/2}\\
&=-\sum_{k=m}^{\infty}\sum_{\ell=0}^{\infty} C_{m,k,\ell}q^{2(k+\ell)}.
\end{split}
\end{align}
Here and below the convention on binomial coefficients removes the endpoint
term without a separate case.

For an even harmonic, reindexing \eqref{r-odd-explicit} gives
\begin{align*}
r_{2m-1} &=-2\sum_{k=m}^{\infty}\binom{-3/2}{2k-1} \binom{2k-1}{k-m}q^{2k-1}A^{-2k-1/2},\\
r_{2m+1} &=-2\sum_{k=m}^{\infty}\binom{-3/2}{2k-1} \binom{2k-1}{k-m-1}q^{2k-1}A^{-2k-1/2}.
\end{align*}
Consequently,
\begin{align}\label{K-even-harmonic-expanded}
\begin{split}
\frac12(r_{2m+1}-r_{2m-1}) &=\sum_{k=m}^{\infty}\binom{-3/2}{2k-1} \left[\binom{2k-1}{k-m}-\binom{2k-1}{k-m-1}\right] q^{2k-1}A^{-2k-1/2}\\
&=\sum_{k=m}^{\infty}\sum_{\ell=0}^{\infty} E_{m,k,\ell}q^{2(k+\ell)-1}.
\end{split}
\end{align}
Since $a$ is even and $\psi$ is odd, only the cosine part of
$\sin n(\theta_0+\psi)$ survives after integration against $a^5$.
Using $q=\varepsilon^2a^2$ and the definitions in \eqref{e.051003},
\eqref{K-first-harmonic-expanded}--\eqref{K-even-harmonic-expanded} give all
three asserted series.  Finally,
\[
C_{0,1,0}=\binom{-3/2}{1} +\binom{-3/2}{2}\left(\binom21-\binom20\right)=\frac38.
\]
\end{proof}

\begin{lemma}\label{l.051002}
The Fourier expansion of ${\mathbb J}$ is
\[
{\mathbb J}=\sum_{m=0}^{\infty}{\mathbb J}^{\odd}_m \sin(2m+1)\theta_0 +\sum_{m=1}^{\infty}{\mathbb J}^{\even}_m \sin2m\theta_0,
\]
where
\begin{align*}
{\mathbb J}^{\odd}_0 &=\sum_{k=1}^{\infty}\sum_{\ell=0}^{\infty} {\mathcal C}_{0,k,\ell}{\mathbb J}_{k+\ell,0,\odd},\\
{\mathbb J}^{\odd}_m &=\sum_{k=m}^{\infty}\sum_{\ell=0}^{\infty} {\mathcal C}_{m,k,\ell}{\mathbb J}_{k+\ell,m,\odd}, \qquad m\geq1,\\
{\mathbb J}^{\even}_m &=\sum_{k=m}^{\infty}\sum_{\ell=0}^{\infty} {\mathcal E}_{m,k,\ell}{\mathbb J}_{k+\ell,m,\even}, \qquad m\geq1.
\end{align*}
For $k\geq1$,
\begin{align*}
{\mathcal C}_{0,k,0} &=3\binom{-5/2}{k}-\binom{-3/2}{k} -\binom{-3/2}{2k} \left[\binom{2k}{k}+\binom{2k}{k-1}\right] -2\binom{-3/2}{2k+1}\binom{2k+1}{k},\\
{\mathcal C}_{0,k,\ell} &=-\binom{-3/2}{2k} \left[\binom{2k}{k}+\binom{2k}{k-1}\right] \binom{-2k-3/2}{\ell}\\
&\quad-2\binom{-3/2}{2k+1}\binom{2k+1}{k} \binom{-2k-5/2}{\ell},\qquad \ell\geq1,\\
{\mathcal C}_{m,k,\ell} &=-\binom{2k+1}{k-m} \left[ \binom{-3/2}{2k}\binom{-2k-3/2}{\ell} +2\binom{-3/2}{2k+1}\binom{-2k-5/2}{\ell} \right],\\
{\mathcal E}_{m,k,\ell} &=\binom{2k}{k-m} \left[ \binom{-3/2}{2k-1}\binom{-2k-1/2}{\ell} +2\binom{-3/2}{2k}\binom{-2k-3/2}{\ell} \right].
\end{align*}
In particular, ${\mathcal C}_{0,1,0}=3/2$.
\end{lemma}

\begin{proof}
Keep the notation \eqref{r0-explicit}--\eqref{r-odd-explicit} and write
\[
B(\phi,q)=1-R_1^{-3}+q\cos\phi(2+R_1^{-3}) =B_0(q)+\sum_{n=1}^{\infty}B_n(q)\cos n\phi.
\]
Multiplication by $\cos\phi$ gives the exact relations
\begin{align}\label{Bn-relation}
\begin{split}
B_0&=1-r_0+\frac q2r_1,\qquad B_1=-r_1+2q+q\left(r_0+\frac12r_2\right),\\
B_n&=-r_n+\frac q2(r_{n-1}+r_{n+1}),\qquad n\geq2.
\end{split}
\end{align}
Because $ab$ is odd, while every $B_n(q)$ is even in $\tau$,
\begin{equation}
\int_{-\infty}^{+\infty}abB_n(q) \cos n(\theta_0+\psi)\,d\tau =-\sin n\theta_0 \int_{-\infty}^{+\infty}abB_n(q)\sin n\psi\,d\tau. \label{J-parity-identity}
\end{equation}
Thus the sine coefficient of ${\mathbb J}$ is the negative of the
corresponding Fourier coefficient $B_n$.

We first treat $n=1$.  From \eqref{r0-explicit}--\eqref{r-odd-explicit},
\begin{align}\label{minus-B1-expanded-raw}
\begin{split}
-B_1 &=3qA^{-5/2}-qA^{-3/2}-2q -q\sum_{k=1}^{\infty}\binom{-3/2}{2k} \left[\binom{2k}{k}+\binom{2k}{k-1}\right] q^{2k}A^{-2k-3/2}\\
&\quad-2\sum_{k=1}^{\infty}\binom{-3/2}{2k+1} \binom{2k+1}{k}q^{2k+1}A^{-2k-5/2}.
\end{split}
\end{align}
The terms of order $q$ cancel, and
\[
3qA^{-5/2}-qA^{-3/2}-2q =\sum_{k=1}^{\infty} \left[3\binom{-5/2}{k}-\binom{-3/2}{k}\right]q^{2k+1}.
\]
Expanding the remaining powers of $A$ in
\eqref{minus-B1-expanded-raw} therefore yields
\begin{equation}
-B_1=\sum_{k=1}^{\infty}\sum_{\ell=0}^{\infty} {\mathcal C}_{0,k,\ell}q^{2(k+\ell)+1}. \label{minus-B1-coeff}
\end{equation}
This also displays directly the cancellation responsible for the order
$q^3$ of the first harmonic.

Let now $n=2m+1\geq3$.  Equations \eqref{Bn-relation},
\eqref{r-even-explicit}, and \eqref{r-odd-explicit} give
\begin{align*}
-B_{2m+1} &=-\sum_{k=m}^{\infty}\binom{2k+1}{k-m}q^{2k+1} \times\left[ \binom{-3/2}{2k}A^{-2k-3/2} +2\binom{-3/2}{2k+1}A^{-2k-5/2} \right].
\end{align*}
Indeed, Pascal's identity combines the two neighboring even coefficients:
\[
\binom{2k}{k-m}+\binom{2k}{k-m-1} =\binom{2k+1}{k-m}.
\]
After expanding the powers of $A$ we obtain
\begin{equation}
-B_{2m+1} =\sum_{k=m}^{\infty}\sum_{\ell=0}^{\infty} {\mathcal C}_{m,k,\ell}q^{2(k+\ell)+1}. \label{minus-Bodd-coeff}
\end{equation}

Finally, for $n=2m$, the two neighboring odd coefficients in
\eqref{Bn-relation} are first reindexed to the common power $q^{2k-1}$.
Pascal's identity then gives
\begin{align}\label{minus-Beven-coeff}
\begin{split}
-B_{2m} &=\sum_{k=m}^{\infty}\binom{2k}{k-m}q^{2k} \times\left[ \binom{-3/2}{2k-1}A^{-2k-1/2} +2\binom{-3/2}{2k}A^{-2k-3/2} \right]\\
&=\sum_{k=m}^{\infty}\sum_{\ell=0}^{\infty} {\mathcal E}_{m,k,\ell}q^{2(k+\ell)}.
\end{split}
\end{align}
Substitution of \eqref{minus-B1-coeff},
\eqref{minus-Bodd-coeff}, and \eqref{minus-Beven-coeff} into the parity
identity \eqref{J-parity-identity}, followed by $q=\varepsilon^2a^2$,
gives the three series in the statement.  Direct substitution at $k=1$
gives
\[
{\mathcal C}_{0,1,0} =3\binom{-5/2}{1}-\binom{-3/2}{1} -\binom{-3/2}{2}\left(\binom21+\binom20\right) -2\binom{-3/2}{3}\binom31=\frac32.
\]
\end{proof}

The following lemma relates the integrals of the homoclinic orbit to integrals involving 
\begin{equation}\label{ztoy1}
y(z) = (\sqrt{z^2 + 4^3}+z)^{1/3} -( \sqrt{z^2 + 4^3} -z)^{1/3},
\end{equation}
and follows directly from a sequence of change-of-variable transformations. 
\begin{lemma}\label{l.043001}
    For integers $p,q\geq 1$ and $r\geq 0$, consider 
    \[I_{p,q,r} := \int_{-\infty}^{\infty}a^pb^re^{\mi q\psi}d\tau.\]
    We have 
    \begin{align*}
I_{p,q,r} = \frac{2}{3}(2 \sqrt{2})^{p+r}(-1)^{r+q}\int_{-\infty}^{\infty}\frac{e^{- \frac{\mi q}{24\varepsilon^3}z}y(z)^r}{(y(z)+2\mi)^{1+q+(p+2r+1)/2}(y(z)-2\mi)^{1-q+(p+2r+1)/2}}dz.
\end{align*}
\end{lemma}
\begin{proof}
    Note that if $ y(x) = x-x^{-1},$ one has 
    \[dy = (1+ x^{-2}) dx = (x + x^{-1}) x^{-1} dx,\]
    and 
    \[y^2 = x^2 +x^{-2}-2 = (x+x^{-1})^2 - 4, \qquad (x+x^{-1})^2 = y^2 + 4.\]
    Furthermore, 
    \[x^3 - x^{-3} = (x-x^{-1})(x^2 + x^{-2} + 1) = y (y^2+3).\]
    Therefore, we have  
    \begin{align*}
&I_{p,q,r} =\int_{\mb R}a^pb^re^{\mi q\psi}d\tau\\
& \overset{x=e^{\tau}}{=}\int_{0}^{+\infty}\frac{(2 \sqrt{2})^{p+r}(x^{-1}-x)^re^{i2q\tan^{-1} \frac{x-x^{-1}}{2}} e^{- \frac{\mi q}{48\varepsilon^3}\left( \left(x^3- x^{-3}\right)+9\left(x- x^{-1}\right)\right)}}{x (x+x^{-1})^{p+2r}} d x\\
&\overset{y=x-x^{-1}}{=} (2 \sqrt{2})^{p+r}\int_{-\infty}^{\infty}\frac{e^{ \mi 2q\tan^{-1} \frac{y}{2}} e^{- \frac{\mi q}{48\varepsilon^3}\left(y(y^2+12)\right)}}{(y^2+4)^{(p+2r+1)/2}}(-y)^rdy\\
&=(2 \sqrt{2})^{p+r}(-1)^r\int_{-\infty}^{\infty}\frac{(2+\mi y)^{2q}(y^2+4)^{-q} e^{- \frac{\mi q}{48\varepsilon^3}\left(y(y^2+12)\right)}}{(y^2+4)^{(p+2r+1)/2}}y^rdy\\
& = (2 \sqrt{2})^{p+r}(-1)^{r+q}\int_{-\infty}^{\infty}\frac{e^{- \frac{\mi q}{48\varepsilon^3}\left(y(y^2+12)\right)}}{(y+2\mi)^{q+(p+2r+1)/2}(y-2\mi)^{-q+(p+2r+1)/2}}y^rdy\\
&\overset{2z=y^3+12y}{=}(2 \sqrt{2})^{p+r}(-1)^{r+q}\int_{-\infty}^{\infty}\frac{e^{- \frac{\mi q}{24\varepsilon^3}z}}{(y+2\mi)^{q+(p+2r+1)/2}(y-2\mi)^{-q+(p+2r+1)/2}}\frac{2y^r}{3(y^2+4)}dz\\
&=\frac{2}{3}(2 \sqrt{2})^{p+r}(-1)^{r+q}\int_{-\infty}^{\infty}\frac{e^{- \frac{\mi q}{24\varepsilon^3}z}y(z)^r}{(y(z)+2\mi)^{1+q+(p+2r+1)/2}(y(z)-2\mi)^{1-q+(p+2r+1)/2}}dz,
\end{align*}
    where we  used the fact that the change of variable $z(y)=\frac12(y^3+12y)$ is a bijection on $\mathbb R$ with inverse \eqref{ztoy1}. 

\end{proof}

\section*{Acknowledgments}
We thank an anonymous referee for valuable comments and suggestions
on an earlier version of this paper.

\end{document}